\documentclass[11pt,reqno]{amsart}
  \usepackage{mathabx}
\usepackage{cite}
\usepackage{color}

\usepackage[margin=1in]{geometry}
\usepackage{appendix}
\allowdisplaybreaks

\theoremstyle{plain}
\newtheorem{thm}{Theorem}[section]

\newtheorem{lem}{Lemma}[section]

\newtheorem{prop}{Proposition}[section]

\theoremstyle{definition}

\theoremstyle{remark}
\newtheorem{rem}{Remark}[section]
\numberwithin{equation}{section}

\newcommand{\R}{\mathbb{R}}
\newcommand{\T}{\mathbb{T}}
\newcommand{\h}{\mathbb{H}}
\newcommand{\cl}{\mathcal{L}}

\newcommand{\ep}{\epsilon}

\begin{document}

\title[Magnetic effect on  solvability of MHD boundary layer equations]
 {Magnetic effects on the solvability of 2D MHD boundary layer equations without resistivity in Sobolev spaces}

 \author[Cheng-Jie Liu]{Cheng-Jie Liu}
\address{Cheng-Jie Liu
\newline\indent
School of Mathematical Sciences, LSC-MOE and Institute of Natural Sciences,
 Shanghai Jiao Tong University,
Shanghai 200240, P. R. China.}
\email{liuchengjie@sjtu.edu.cn}

\author[Dehua Wang]{Dehua Wang}
\address{Dehua Wang
\newline\indent
\small Department of Mathematics,
University of Pittsburgh, Pittsburgh, PA 15260, USA.}
\email{dwang@math.pitt.edu}

\author[Feng Xie]{Feng Xie}
\address{Feng Xie
\newline\indent
School of Mathematical Sciences, and LSC-MOE,
 Shanghai Jiao Tong University,
Shanghai 200240, P. R. China.}
\email{tzxief@sjtu.edu.cn}

\author[Tong Yang]{Tong Yang}
\address{Tong Yang
\newline\indent
Department of Mathematics,
City University of Hong Kong,
Tat Chee Avenue, Kowloon, Hong Kong; and
School of Mathematical Sciences, Shanghai Jiao Tong University,
Shanghai 200240, P. R. China.}
\email{matyang@cityu.edu.hk}

\begin{abstract}
In this paper, we are concerned with the magnetic effect on  the Sobolev solvability of boundary layer equations for the 2D incompressible MHD system without resistivity.
The MHD boundary layer is described by the Prandtl type equations derived from the incompressible viscous
MHD system without resistivity  under the no-slip boundary condition on the velocity. Assuming that the initial tangential magnetic field does not degenerate, a
local-in-time  well-posedness  in Sobolev spaces is proved without the monotonicity condition on the velocity field. Moreover, we show that if the tangential magnetic field shear layer is degenerate at one point, then the linearized MHD boundary layer system  around the shear layer profile is ill-posed in the Sobolev settings provided that the initial velocity shear flow is non-degenerately critical at the same point.
\end{abstract}

\date{\today}


\maketitle


\section{Introduction and Main Result} \label{S1}
In this paper, we consider the initial-boundary value problem for the following two-dimensional (2D) magnetohydrodynamic (MHD) boundary layer equations in the domain $\{(t,x,y): t\in[0,T],x\in\T,y\in\R_+\}$:
\begin{align}
\label{BLE}
\left\{
\begin{array}{ll}
\partial_tu+u\partial_xu+v\partial_{y}u+\partial_x p-\partial_{y}^2u-b_1\partial_xb_1-b_2\partial_{y}b_1=0,\\
\partial_t b_1+u\partial_x b_1+v\partial_{y}b_1-b_1\partial_xu-b_2\partial_{y}u=0,\\
\partial_xu+\partial_{y}v=0,\quad \partial_xb_1+\partial_{y}b_2=0,\\
(u,v,b_2)|_{y=0}=0,\qquad \lim\limits_{y\to+\infty}(u,b_1)(t,x,y)=(U,B)(t,x),\\
(u,b_1)|_{t=0}=(u_0,b_0)(x,y),
\end{array}
\right.
\end{align}
where $\T$ {stands for a torus or a periodic domain, $\R_+=[0, +\infty)$}, $(u,v)$ and  $(b_1, b_2)$ are the velocity and magnetic boundary layer functions respectively, and the known functions $U, B$ and $p$ satisfy the Bernoulli law:
\begin{align}\label{BN1}
\begin{cases}
&\partial_t U+U\partial_x U+\partial_x p=B\partial_x B,\\
&\partial_t B+U\partial_x B=B\partial_x U.
\end{cases}\end{align}
See the Appendix \label{A1} for the derivation of the system \eqref{BLE}.

Before stating the main results in this paper, we first review some related works on the Prandtl boundary layer theories. In fact, without  the magnetic field $(b_1, b_2)$ in \eqref{BLE},  the system is  the classical Prandtl equations that was  firstly derived by L. Prandtl \cite{P} in 1904 to understand the structure of incompressible fluid with high Reynolds number and physical boundaries.  In the 2D case,  {when the initial tangential velocity satisfies the monotonicity assumption, Oleinik \cite{O,O-S} firstly achieved the local-in-time well-posedness of classical solutions
by using the Crocco transformation, and this well-posedness result was recently  reproved by   an energy method  in the framework of
weighted Sobolev spaces in \cite{AWXY} and \cite{M-W1} independently, where the cancellation mechanism in the convection terms is observed and essentially used. Also a global in time weak solution was obtained in \cite{X-Z}   under an additional favorable condition on the pressure.}

 {On the contrary, when the monotonicity condition is violated, separation of the boundary
layer or singularity formation 
is well expected and observed. For example, E-Engquist  in \cite{E-E} proved that the smooth solution to the Prandtl
equations must blowup in a finite time. Very recently, when the background shear layer admits a non-degenerate critical point, the ill-posedness (or instability phenomena) of solutions to both the linearized and nonlinear Prandtl equations was studied, see
\cite{G-D,G-N,G,G-N1, L-Y} and the reference therein. All these results show that the monotonicity condition plays a key role for the well-posedness theory of solutions in the finite regularity functional spaces. However, this is not the case in the frameworks of analytic functions and Gevrey regularity classes. In fact, in the framework of
analytic functions, Sammartino and Caflisch \cite{S-C1,S-C2} not only established
the local well-posedness theory of the Prandtl system, but also proved the validity of Prandtl boundary layer ansatz in this setting by using
the abstract Cauchy-Kowalewskaya theorem.} 
We refer the readers to \cite{ L-C-S, I-V, K-V, K-M-V-W, Mae, N-N, W-W-Z, Z-Z} and the reference therein  for more results in the analytic framework,  and   \cite{D-G, G-M-M, G-M, L-W-X, L-Y} in the Gevery framework. It is noted that the above results   mainly concentrated on the two-dimensional case, and there are only a few results in the  three-dimensional case such as \cite{F-T-Z, L-W-Y, L-W-Y1, L-W-Y2, L-X}.

Motivated by the fifteenth open problem in Oleinik-Samokhin's classical book \cite{O-S} (page 500-503),
{\it ``15. For the equations of the magnetohydrodynamic boundary layer, all problems of the above type are still open,"}
efforts have been made to study the well-posedness of solutions to the MHD boundary layer equations and to justify the MHD boundary layer ansatz in \cite{H-L-Y,L-X-Y1,L-X-Y2, X-Y1, X-Y2}; see also   \cite{G-P}. Precisely, when the hydrodynamic  and magnetic Reynolds numbers have the same order, the well-posedness of solutions to the MHD boundary layer equations and the validity of the Prandtl ansatz were established without any monotone condition imposed on the  velocity in \cite{L-X-Y1,L-X-Y2}. And the long-time existence of solutions to the MHD boundary layer equations in analytic settings for two different physical regimes were also studied in \cite{X-Y1, X-Y2}. When the magnetic Reynolds number is much larger than the hydrodynamic Reynolds number,  the resistivity terms can be ignored in the MHD equations. As a consequence, there is no partial viscous effect in normal variable $y$ for the second equation in (\ref{BLE}).

The purpose of this paper is to study the well-posedness and ill-posedness of \eqref{BLE}.
First, similarly to \cite{L-X-Y1},  we shall establish a well-posedness theory
for the MHD boundary layer equations (\ref{BLE})   in weighted Sobolev spaces, provided that the initial tangential magnetic filed is not degenerate.
 This shows that the tangential magnetic field prevents the formation of  singularity in more general
 flow situation that includes  reverse flow in the velocity field,  {no matter whether there is partial viscous effect in the magnetic boundary layer equation or not}.
 This  result on the well-posedness   can be stated as follows.

\begin{thm}
\label{MTH}
Suppose that the outflow $(U, p, B)(t,x)$ in (\ref{BN1}) 
is smooth, and the initial data and boundary conditions in \eqref{BLE} are smooth, compatible and satisfy
\begin{align}
\label{CC}
b_{0}(x,y)\geq \delta_0
\end{align}
for some positive constant $\delta_0$.
Then there exists a time $T$, such that the initial-boundary value problem \eqref{BLE} 
admits a unique smooth solution $(u, v, b_1, b_2)(t,x, y)$ satisfying
\begin{align}
\label{LB}
b_{1}(t,x,y)\geq \delta_0/2
\end{align}
for all $t\in[0,T], (x, y)\in\mathbb{T}\times\mathbb{R}_+$.
\end{thm}
We remark that  the precise smoothness condition and the compatibility condition in the above Theorem \ref{MTH} will be given later for the concise  presentation of the theorem.
Theorem \ref{MTH} shows that the non-degenerate tangential magnetic field has stabilizing effect on boundary layers. On the other hand, when the magnetic field is absent, the Prandtl equations exhibit instability mechanism in the framework of Sobolev spaces without the monotonicity condition \cite{G-D,G-N,G-N1, L-Y}. Then a natural question arises as whether such a non-degeneracy condition on the tangential magnetic field is necessary for the well-posedness of the system  \eqref{BLE}. Our second result
in this paper aims to answer this question.

For illustration and without loss of generality, we consider \eqref{BLE} with the constant outflow $(U,B)$. That is,
\[
(U,B)(t,x)~\equiv~(U_0,B_0)\quad\mbox{so that},\quad \partial_x p(t,x)~\equiv~0.
\]
In this case,  the equations of \eqref{BLE} admit a  shear flow solution of the form
$$\big(u,v,b_1,b_2\big)(t,x,y)~\equiv~\big(u_s(t,y),0,b_s(y),0\big),$$
where the function $u_s(t,y)$ is a smooth solution to the following heat equation:
\begin{equation}\label{heat}\begin{cases}
\partial_t u_s-\partial_y^2u_s=0,\qquad {\mbox {in}}\quad\Omega,\\
u_s|_{y=0}=0,\quad\lim\limits_{y\to+\infty} u_s=U_0,\\
u_s|_{t=0}=U_s(y),
\end{cases}\end{equation}
with an initial shear layer $U_s(y)$.
Consider the linearization of
the problem \eqref{BLE} around the shear flow $\big(u_s(t,y),0, b_s(y), 0\big)$ in the domain $\{(t,x,y): t\in[0,T],x\in\T,y\in\R_+\}$, we obtain
\begin{equation}\label{linear_pr}\begin{cases}
\partial_t u+u_s\partial_x u+v\partial_y u_s-\partial_y^2u-b_s\partial_x b_1-b_2b_s'=0,\qquad &\\
\partial_t b_1+u_s\partial_x b_1+v b_s'-b_s\partial_x u-b_2\partial_y u_s=0,\qquad &\\
\partial_x u+\partial_y v=0,\qquad \partial_x b_1+\partial_y b_2=0, \\
(u, v, b_2)|_{y=0}=0,\quad \lim\limits_{y\to\infty} (u,b_1)(t,x,y)=\mathbf0.
\end{cases}\end{equation}

For any $\alpha,s\geq0$, denote
\[\begin{split}
&W_\alpha^{s,\infty}(\mathbb{R}_+)~:=~\{f=f(y), y\in\mathbb{R}_+: ~\|f\|_{W_\alpha^{s,\infty}}= 
\|e^{\alpha y}f(y)\|_{W^{s,\infty}(\mathbb{R}_+)}<\infty\}
\end{split}\]
 and
\begin{align*}
E_{\alpha,\beta}=\big\{f=f(x,y)=\sum_{k\in\mathbb{Z}}e^{ikx}f_k(y): \ \|f_k\|_{W_{\alpha}^{0,\infty}}\leq C_{\alpha,\beta}e^{-\beta|k|},\ \forall k\in\mathbb{Z}\big\}
\end{align*}
with
\begin{align*}
\|f\|_{E_{\alpha,\beta}}:=\sup_{k\in\mathbb{Z}}e^{\beta|k|}\|f_k\|_{W_{\alpha}^{0,\infty}(\R_+)}.
\end{align*}
It is noted that functions in $E_{\alpha,\beta}$ have analytic regularity in $x$-variable. 
The following proposition states the well-posedness of solutions in the analytic function spaces to (\ref{linear_pr}) in  $E_{\alpha,\beta}$. 
\begin{prop} 
\label{PRO1}
Suppose $u_s-U_0\in C\big(\mathbb{R}^+; W^{1,\infty}_{\alpha}(\R_+)\big)$ and $b_s-B_0\in W^{1,\infty}_{\alpha}(\R_+)$. Then, there exists a constant $\delta>0$ such that for any $T>0$ with $\beta-\delta T>0$, and for any
$(u_0, b_0)(x,y)\in E_{\alpha,\beta}$, the linearized problem (\ref{linear_pr}) admits a unique solution  satisfying
\begin{align*}
 \big(u, b_1\big)(t,x,y)\in C\big([0,T); E_{\alpha,\beta-\delta T}\big)\quad with\quad (u, b_1)|_{t=0}=\big(u_0, b_0\big)(x,y).
\end{align*}
\end{prop}
This proposition shows that the linearized problem (\ref{linear_pr}) is well-posed in the analytic setting, at least in local time. Its proof is similar to the  Proposition 1 in \cite{G-D} and the Proposition 1.1 in \cite{L-X-Y3}, thus  is omitted for brevity.

Denoted by $\mathcal{T}(t,s)$ the solution operator of \eqref{linear_pr}
\begin{equation}\label{def_T}
\mathcal{T}(t,s)(u_0, b_0)~:=~(u,b_1)(t,\cdot),
\end{equation}
where $(u,b_1)$ is the solution to the problem \eqref{linear_pr} with $(u, b_1)|_{t=s}=(u_0, b_0).$
Since the space $E_{\alpha,\beta}$ is dense in the Sobolev type space
\[
\h_\alpha^m~:=~\{f=f(x,y),(x,y)\in\mathbb{T}\times\R_+: ~\|f\|_{\h_\alpha^m}= 
\|f(\cdot)\|_{H^m(\mathbb{T},W_\alpha^{0,\infty}(\mathbb{R}_+))}<\infty\},\quad m\geq0,
\]
by Proposition \ref{PRO1} we can define for all $m_1, m_2\geq0,$
\begin{align*}
\displaystyle \|\mathcal{T}\|_{\mathcal{L}(\h_\alpha^{m_1},\h_\alpha^{m_2})}~=~\sup_{(u_0,b_0)\in E_{\alpha,\beta}}\frac{\|\mathcal{T}(u_0,b_0)\|_{\h^{m_2}_\alpha}}{\|(u_0,b_0)\|_{\h^{m_1}_\alpha}}\in \R_+\cup\{+\infty\}.
\end{align*}
Motivated by the paper \cite{G-D} on the instability of solutions to the linearized Prandtl equations,
the following result shows that when the background magnetic field profile degenerates at the non-degenerate critical point of the background velocity shear layer,
  the linearized problem \ref{linear_pr} is ill-posed  in Sobolev spaces.
\begin{thm}\label{thm_lin}
Let $u_s(t,y)$ be the solution of   \eqref{heat} satisfying
 $$u_s-U_0\in C\big(\R_+; W_\alpha^{4,\infty}(\R_+)\big)\cap C^1\big(\R_+; W_\alpha^{2,\infty}(\R_+)\big),$$
and $b_s(y)-B_0\in W_\alpha^{4,\infty}(\R_+)$.
Assume  that $b_s(y)$ and the initial shear layer $U_s(y)$ satisfy
 \[\exists ~a>0, \quad\textit{s.t.}\quad b_s(a)=b_s'(a)=U_s'(a)=0,\quad U_s''(a)\neq0. \]
Then, there exists $\sigma>0$ such that for any $\delta>0$,
\begin{equation}\label{est_in}
\sup\limits_{0\leq s< t\leq\delta}\big\|e^{-\sigma(t-s)\sqrt{|\partial_x|}}\mathcal{T}(t,s)\big\|_{\cl(\h_\alpha^m,\h_\alpha^{m-\mu})}
~=~+\infty,\quad \forall \alpha,m\geq0,~\mu\in[0,\frac{1}{4}).
\end{equation}
\end{thm}
\begin{rem}
\label{rem1}
 {We remark that the result in Theorem \ref{thm_lin} improves the previous work \cite{L-X-Y3} significantly. In \cite{L-X-Y3}  the partial viscous term $\partial_y^2b_1$ is included in the second equation of (\ref{BLE}). Consequently, the background shear flow solution should take the form of $\big(u_s(t,y),0, b_s(t,y), 0\big)$ in \cite{L-X-Y3}, instead of $\big(u_s(t,y),0, b_s(y), 0\big)$. Under the same conditions on the initial data of $(u_s(t,y), b_s(t,y))$ as those in Theorem \ref{thm_lin}:
\[\exists ~a>0, \quad\textit{s.t.}\quad b_s(0, a)=\partial_yb_s(0, a)=U_s'(a)=0,\quad U_s''(a)\neq0, \]
  the same conclusion (\ref{est_in}) still holds for the linearized problem (\ref{linear_pr}) with additional partial diffusion term $\partial_y^2b_1$ in the second equation of (\ref{linear_pr}).  {More precisely, to obtain the same results in (\ref{est_in}), the assumptions $\partial^i_yb_s(0, a)=0\ (i=0,1,...,6)$ are required in \cite{L-X-Y3}, which can be relaxed to the assumptions   $\partial^i_yb_s(0, a)=0\ (i=0,1,2)$ by the construction of approximate solutions proposed in this paper.} }
\end{rem}

Comparing the previous results obtained in \cite{L-X-Y1, L-X-Y3} with the results stated in Theorems \ref{MTH} and \ref{thm_lin}, we find that the partial viscous term $\partial_y^2b_1$ in the magnetic boundary layer equation has only few effects on both the well-posedness of solution to the MHD boundary layer equations (\ref{BLE}) and the ill-posednes of solutions to the linearized problem (\ref{linear_pr}). This is indeed one of the observations in this paper.  {More precisely, when there is no partial viscous term $\partial_y^2b_1$ in MHD boundary layer equations, we can still   establish the same well-posedness theory of solutions as what was proved in \cite{L-X-Y1} provided that the magnetic field does not degenerate. And when the magnetic field degenerates in the sense of   Remark \ref{rem1}, we can show the same ill-posedness result as what was achieved in \cite{L-X-Y3}. However, different from the energy method used in \cite{L-X-Y1}, to prove the well-posedess theory, we need to use some equivalent Sobolev spaces and the induction method to overcome the difficulties caused by the absence of partial diffusion term $\partial_y^2b_1$. In addition, as what is stated in Remark \ref{rem1}, we find a new construction of growing modes, which can essentially relax the assumptions required in \cite{L-X-Y3}.}

The rest of the paper is organized as follows. In Sections 2 and 3 , we will prove
the well-posedness of  (\ref{BLE}) in  Sobolev spaces when the tangential magnetic field is not degenerate. Precisely, in Section 2, we reformulate the boundary layer system into a model similar to the model of Chaplygin gas by  variable transformation. In  Section 3, we  establish the well-posedness theory for the reduced Chaplygin type model that leads to  the well-posedness of the original boundary layer system. In  Section 4, we will prove Theorem \ref{thm_lin} about the linear instability of \eqref{linear_pr} in Sobolev spaces when the tangential magnetic field is degenerate at one point. Finally, the formal derivation of MHD boundary layer problem (\ref{BLE}) is given in the Appendix A. The Appendix B gives the proof of Lemma \ref{MTV}.

\bigskip

\section{Reformulation of the  system}

To establish the well-posedness  of the boundary layer problem \eqref{BLE}, the main difficulty comes from the loss of $x-$derivatives in the normal components $v$ and $b_2$. To overcome this difficulty,  inspired by \cite{L-X-Y1}, from the divergence free condition of the magnetic field in $(\ref{BLE})$, we introduce the stream function $\psi$ of magnetic field $(b_1, b_2)$, such that
\begin{align*}
\partial_{y}\psi=b_1,\qquad-\partial_x\psi=b_2.
\end{align*}
It follows from the second equation in \eqref{BLE} that the stream function $\psi$ satisfies the transport equation:
\begin{align}
\label{EP}
\partial_t\psi+u\partial_x\psi+v\partial_{y}\psi=0,
\end{align}
with the boundary condition
\begin{align}
\psi|_{y=0}=0.
\end{align}
Its initial data is given by  the initial data $b_{0}$,
\begin{align}
\label{IG}
\psi(0,x,y)=\int_0^{y}b_{0}(x,s)ds.
\end{align}
From $(\ref{CC})$, one has
\begin{align}
\label{ASSUM1}
\partial_y\psi(0,x,y)= b_{0}(x,y) \geq \delta_0,
\end{align}
which implies that $\psi(0,x,y)\geq 0$ is an increasing function with respect to $y$ with
\begin{align*}
\lim_{y\rightarrow+\infty}\psi(0,x,y)=+\infty.
\end{align*}
Then, the transport equation (\ref{EP}) and the initial data (\ref{IG}) yield that $\psi(t,x,y)\geq 0$ is an increasing function with respect to $y$ for every $(t,x)\in[0,T]\times\mathbb{T}$, and
\begin{align*}
\lim_{y\rightarrow+\infty}\psi(t,x,y)=+\infty,
\end{align*}
provided that $u$ and $v$ are Lipschitz continuous functions.

In this way, we can introduce a new coordinate transformation,
\begin{align}
\label{NV}
\bar{t}=t,\quad\bar{x}=x,\quad\bar{y}=\psi(t,x,y).
\end{align}
In the new coordinates (\ref{NV}), the region $\{(t,x,y): t\in[0,T],x\in\T,y\in\R_+\}$ is mapped into $\{(\bar t,\bar x,\bar y): \bar t\in[0,T],\bar x\in\T,\bar y\in\R_+\}$, and the boundary of $\{y=0\}$ ($\{y=+\infty\}$ respectively)  becomes the boundary of $\{\bar y=0\}$ ($\{\bar y=+\infty\}$ respectively). Also, the equations in (\ref{BLE}) can be written as
\begin{align}
\label{BLEE}
\left\{
\begin{array}{ll}
\partial_{\bar{t}}u+u\partial_{\bar{x}}u+\partial_{\bar{x}}p-b_1\partial_{\bar{y}}(b_1\partial_{\bar{y}}u)-b_1\partial_{\bar{x}}b_1=0,\\
\partial_{\bar{t}}b_1+u\partial_{\bar{x}}b_1-b_1\partial_{\bar{x}}u=0,
\end{array}
\right.
\end{align}
with the boundary condition
\begin{align}
\label{BCP}
u|_{\bar{y}=0}=0,
\end{align}
and the far-field condition
\begin{align}
\label{FCP1}
\lim_{\bar{y}\rightarrow+\infty}u=U,\quad \lim_{\bar{y}\rightarrow+\infty}b_1=B.
\end{align}
Without any confusion, we still denote the initial data by $(u_0, b_0)=(u_0, b_0)(\bar{x}, \bar{y})$.

It is noted that there are no normal components of velocity and magnetic field in the reduced equations.
In the subsequent section,  we will study  (\ref{BLEE})-(\ref{FCP1}) with initial data $(u_0, b_0)$ in
some Sobolev spaces, which in turn yields a corresponding result on the original problem (\ref{BLE}).

\begin{rem}
If one sets $\rho=1/b_1$, then the equations \eqref{BLEE} in terms of $(\rho, u)$ become
\begin{align}
\label{BLLE}
\left\{
\begin{array}{ll}
\partial_{\bar{t}}\rho+u\partial_{\bar{x}}\rho+\rho\partial_{\bar{x}}u=0,\\
\rho\partial_{\bar{t}}u+\rho u\partial_{\bar{x}}u+\rho\partial_{\bar{x}}p+\partial_{\bar{x}}(-\rho^{-1})=\partial_{\bar{y}}(\rho^{-1}\partial_{\bar{y}}u),
\end{array}
\right.
\end{align}
which is related to a model of Chaplygin gas.
\end{rem}

\bigskip

\section{Well-posedness of Solutions}

In this section, we will study the initial-boundary value problem (\ref{BLEE})-(\ref{FCP1}) and prove the well-posedness of solutions in Theorem \ref{MTH}.
 Without any confusion, we replace the notation of
$(\bar{t}, \bar{x}, \bar{y})$ by $(t,x,y)$ and write (\ref{BLEE})-(\ref{FCP1}) as
\begin{align}
\label{BLLEE}
\left\{
\begin{array}{ll}
\partial_{t}b_1+u\partial_{x}b_1-b_1\partial_{x}u=0,\\
\partial_{t}u+u\partial_{x}u+\partial_{x}p-b_1\partial_{y}(b_1\partial_{y}u)-b_1\partial_{x}b_1=0,
\end{array}
\right.
\end{align}
with the initial-boundary conditions
\begin{align}
\label{BCP}
(b_1, u)|_{t=0}=(b_0, u_0)(x,y),\qquad u|_{y=0}=0,
\end{align}
and the far-field conditions
\begin{align}
\label{FCP}
\lim_{y\rightarrow+\infty}b_1=B, \quad \lim_{y\rightarrow+\infty}u=U.
\end{align}
Denote
\begin{align*}
A_0(b_1,u)=\left(
\begin{array}{ll}
b_1^{-1},&0\\
0,&b_1^{-1}
\end{array}
\right),\qquad
A_1(b_1,u)=\left(
\begin{array}{ll}
b_1^{-1}u,&-1\\
-1,&b_1^{-1}u
\end{array}
\right),
\end{align*}
and
\begin{align*}
B(b_1,u)=\left(
\begin{array}{ll}
0,&0\\
0,&b_1
\end{array}
\right).
\end{align*}
Then the system (\ref{BLLEE}) can be formulated as the following quasi-linear symmetrical system with low order term and partial diffusivity,
\begin{align}
\label{VBE}
A_0(b_1,u)\partial_t\left(
\begin{array}{c}
b_1\\
u
\end{array}
\right)
+
A_1(b_1,u)\partial_x\left(
\begin{array}{c}
b_1\\
u
\end{array}
\right)
=\partial_y\left(B(b_1,u)\partial_y\left(
\begin{array}{c}
b_1\\
u
\end{array}
\right)\right)
-\left(
\begin{array}{c}
0\\
b_1^{-1}\partial_xp
\end{array}
\right).
\end{align}
Here $A_0$ is a positive definite, symmetric matrix, provided that $b_1>0$, and the matrix $A_1$ is symmetric.

To state the main result in this section, we define some function spaces. Set
\begin{align*}
\Omega=\{(x,y): x\in\mathbb{T}, y\in\mathbb{R}_+\},
\end{align*}
and
\begin{align*}
\Omega_T=[0,T]\times\Omega=\{(t,x,y): t\in[0,T], x\in\mathbb{T}, y\in\mathbb{R}_+\}.
\end{align*}
Denote by $H^k(\Omega)$ the classical Sobolev spaces of  function $f\in H^k(\Omega)$ such that 
\[
\|f\|_{H^k(\Omega)}:=\left(\sum_{\alpha_1+\alpha_2\leq k}
\|\partial^{\alpha_1}_x\partial^{\alpha_2}_y f\|^2_{L^2(\Omega)}\right)^\frac{1}{2}<\infty.
\]
The derivative operator with multi-index is denoted as
\begin{align*}
\partial^\alpha=\partial_t^{\alpha_1}\partial_x^{\alpha_2}\partial_y^{\alpha_3},\quad \alpha=(\alpha_1, \alpha_2, \alpha_3)\in\mathbb{N}^3\quad\mbox{with}\quad |\alpha|=\alpha_1+\alpha_2+\alpha_3.
\end{align*}
The Sobolev space and norm are thus defined as
\begin{align*}
\mathcal{H}^m(\Omega_T)=\{f(t,x,y):~\|f\|_{\mathcal{H}^m(\Omega_T)}= 
\sup_{0\leq t< T}\|f(t)\|_{\mathcal{H}^m}<\infty\}
\end{align*}
with
\begin{align*}
\|f(t)\|_{\mathcal{H}^m}= 
\left(\sum_{|\alpha|\leq m}\|\partial^\alpha f(t,\cdot)\|_{L^2(\Omega)}^2\right)^{1/2}.
\end{align*}
Similarly, we denote the tangential derivative operator
\begin{align*}
\partial_{\tau}^\alpha=\partial_t^{\alpha_1}\partial_x^{\alpha_2},\quad \alpha=(\alpha_1, \alpha_2)\quad \hbox{with}\quad |\alpha|=\alpha_1+\alpha_2,
\end{align*}
and define the following non-isotropic Sobolev space
\begin{align*}
\|f(t)\|_{\mathcal{B}^{k_1,k_2}}=\left(\sum_{|\alpha| \leq k_1, 0\leq q\leq k_2}\|\partial_\tau^\alpha\partial_y^q f(t,\cdot)\|^2_{L^2(\Omega)}\right)^{1/2}.
\end{align*}
Note  that
\begin{align}\label{norm-equ}
\left\|f(t)\right\|_{\mathcal{H}^m}^2=\sum_{j=0}^m\left\|f(t)\right\|_{\mathcal{B}^{m-j,j}}^2.
\end{align}
Moreover, we define
\begin{align*}
\displaystyle \|f(t)\|_{\mathcal{C}^k}=\sum_{|\alpha|+q\leq k}\|\partial_\tau^{\alpha}\partial_y^{q}f(t,\cdot)\|_{L^2\big(\T_x; L^\infty(\R_{+,y})\big)}. 
\end{align*}
%
We shall use the notation $A\lesssim B$   meaning   $|A|\leq C |B|$ with a generic constant $C>0$.

The following 
inequality will be used frequently and its proof will be given in Appendix B. 
\begin{lem}\label{MTV}
For   suitable functions $u$ and $v$, 
it holds that for any $\alpha, \beta\in\mathbb{N}^3, |\alpha|+|\beta|\leq m$ with $m\geq2,$
\begin{align}\label{Moser}
\|(\partial^\alpha u\cdot \partial^\beta v )(t,\cdot)\|_{L^2(\Omega)}\lesssim \|u(t)\|_{\mathcal{H}^m}\|v(t)\|_{\mathcal{H}^m},
\end{align}
and
\begin{align}\label{Moser1}
\|(\partial^\alpha u\cdot \partial^\beta v )(t,\cdot)\|_{L^2(\Omega)}\lesssim \|u(t)\|_{\mathcal{C}^m}\|v(t)\|_{\mathcal{H}^m}.
\end{align}
\end{lem}

Now, we state the main result in this section.
\begin{thm}[Local existence]
\label{T3.1}
 Let $m\geq 4$, and suppose that the trace $(U, B, p)(t,x)$ of the outflow satisfies
\begin{align}\label{OUTE}
\sup_{0\leq t\leq T}\sum_{|\alpha|\leq2m}\|\partial_\tau^\alpha (U, B, p)(t,\cdot)\|_{L^2(\mathbb{T}_x)}\leq M
\end{align}
for some positive constant $M$. Assume that the initial data $(u_0, b_0)$ satisfies
\begin{align}
 \label{IID}
\big( b_0-B(0,x), u_0-U(0,x)\big)\in {H}^{3m}(\Omega),
\qquad b_0\geq \delta_0
 \end{align}
 for some positive constant $\delta_0$, and satisfies the compatibility conditions up to the $m-$th order for the  initial-boundary problem (\ref{VBE}) with (\ref{BCP})-(\ref{FCP}).
Then there exists a positive $T_*$, such that the problem (\ref{VBE}), (\ref{BCP})-(\ref{FCP}) admits a unique solution $(b_1, u)$ satisfying
\begin{align}\label{positi}
b_1(t,x,y)\geq \delta_0/2,\quad (t,x,y)\in\Omega_{T_*},
\end{align}
and
\begin{align}
b_1-B(t,x)\in\mathcal{H}^m(\Omega_{T_*}),\qquad u-U(t,x)\in\mathcal{H}^m(\Omega_{T_*}),\qquad \partial_y u \in L^2(0,T_*; \mathcal{H}^m).
\end{align}
\end{thm}

\begin{rem}
 {The requirements of initial time regularity of solutions can be changed into the requirements of space regularity of initial data in \eqref{IID} through the equations \eqref{BLLEE}}.  {This is the reason why we require the Sobolev space index to be $3m$ in \eqref{IID}.}
\end{rem}
To prove Theorem \ref{T3.1}, we will use the Picard iteration  and the fixed point theorem. For this purpose, we first homogenize the boundary and far-field conditions. Precisely, let $\phi(y)$ be a smooth function satisfying $0\leq\phi(y)\leq 1$ and
\begin{align*}
\phi(y)=\left\{
\begin{array}{ll}
0, &0\leq y\leq 1,\\
1, &y\geq 2.
\end{array}
\right.
\end{align*}
Write
\begin{align*}
b_1=b+B(t,x),\qquad u=v+U(t,x)\phi(y),
\end{align*}
and from \eqref{BCP}-\eqref{FCP} we obtain the following  boundary  and far-field conditions,
\begin{align*}
	v|_{y=0}=0,\quad \lim\limits_{y\rightarrow+\infty}(b, v)(t,x,y)=\mathbf0,
\end{align*}
and the  initial data
\begin{align}\label{IV}
	(b, v)(0,x,y)=\Big(b_0(x,y)-B(0,x), u_0(x,y)-U(0,x)\phi(y)\Big).
\end{align}
From \eqref{BLLEE}, it follows that
\begin{align}
\label{NLP}
\left\{
\begin{array}{ll}
\partial_t b+(v+U\phi)\partial_x b-(b+B)\partial_x v+B_x v-\phi U_x b=r_1,\\
\partial_tv-(b+B)\partial_xb+(v+U\phi)\partial_xv-(b+B)\partial_y\big((b+B)\partial_y v+\phi_y Ub\big)-(\phi_{yy} UB+B_x) b+\phi U_x v =r_2,
\end{array}
\right.
\end{align}
with
\begin{align*}
\begin{cases}
	r_1=-B_t-\phi UB_x+\phi BU_x=(1-\phi)(UB_x-BU_x),\\
	r_2=-\partial_x p-\phi U_t+BB_x-\phi^2U U_x+\phi_{yy} UB^2=(1-\phi)U_t+(1-\phi^2)UU_x+\phi_{yy} UB^2,
\end{cases}
\end{align*}
where we have used \eqref{BN1}. Moreover, by the assumptions in \eqref{OUTE}, it holds
\begin{align}\label{est-r}
	\|(r_1, r_2)\|_{\mathcal{H}^m(\Omega_T)}\leq CM^3
\end{align}
for some constant $C>0$.

Next, set
\begin{align}\label{initial-j}
  (b_0^j, v_0^j)(x,y)~:=~(\partial_t^j b, \partial_t^j v)(0,x,y),\quad 0\leq j\leq m.
\end{align}
It follows from the assumptions in Theorem \ref{MTV} that $(b_0^j, v_0^j)(x,y)\in H^m(\Omega)$, which can be derived from the equations \eqref{NLP} and initial values $(b, v)(0,x,y)$ of \eqref{IV} by induction with respect to $j$.  Moreover, the assumptions in Theorem \ref{MTV} imply that there exists a positive constant $M_0>1$, depending only on $M$ and $(b, v)(0,x,y)$, such that
\begin{align}\label{regular-t}
  \sum_{j=0}^{m}\big\|(b_0^j, v_0^j)\big\|_{H^{3m-2j}(\Omega)}\leq M_0.
\end{align}

\subsection{Construction of approximate solution sequence}\label{SUB}
In this subsection, we will construct an approximate solution sequence $\displaystyle \{{\bf v}^n(t,x,y)\}_{n=0}^\infty=\{(b^n, v^n)(t,x,y)\}_{n=0}^\infty$ to the system \eqref{NLP}.

Firstly, denote
\begin{align*}
{\bf v}_0^j(x,y)~:=~(b_0^j, v_0^j)(x,y),
\end{align*}
and define the zero-th approximate solution ${\bf v}^0$ of \eqref{NLP}  as the following,
\begin{align}\label{iteration-0}
{\bf v}^0(t,x,y)=(b^0, v^0)(t,x,y)~:=~\sum_{j=0}^m\frac{t^j}{j!}{\bf v}_0^j(x,y).
\end{align}
It is straightforward  to check that
\begin{align}
{\bf v}^0(t,x,y)\in \mathcal{H}^m(\Omega_{T}).
\end{align}
Moreover, there exists a $0<T_0\leq T$ such that
\begin{align*}
	b^0(t,x,y)+B(t,x)\geq\frac{\delta_0}{2},\quad \forall\ (t,x,y)\in\Omega_{T_0}.
\end{align*}

Next, we construct  ${\bf v}^{n+1}=(b^{n+1}, v^{n+1}), ~n\geq0$ by induction. Precisely, suppose that the $n-$th order approximate solution ${\bf v}^n=(b^n, v^n)\in \mathcal{H}^m(\Omega_{T_n}), n\geq0$ is obtained for some $0<T_n\leq T_0$ 
and satisfies
\begin{align}\label{iteration-n}
	b^n(t,x,y)+B(t,x)\geq\frac{\delta_0}{2},\quad \partial_t^j{\bf v}^n(0,x,y)={\bf v}_0^j(x,y),\qquad \mbox{for}~(t,x,y)\in\Omega_{T_n},~ 0\leq j\leq m.
\end{align}
We define ${\bf v}^{n+1}(t,x,y)=(b^{n+1}, v^{n+1})(t,x,y)$ by solving the following linear initial-boundary value problem, {\it i.e.}, the iteration scheme of nonlinear problem (\ref{NLP}),
\begin{align}\label{iteration-n+1}
\displaystyle \left\{
\begin{array}{ll}
\partial_t b^{n+1}+(v^n+U\phi)\partial_x b^{n+1}-(b^n+B)\partial_x v^{n+1}+B_xv^{n+1}-\phi U_xb^{n+1} =r_1,\\
\partial_t v^{n+1}-(b^n+B)\partial_x b^{n+1}+(v^n+U\phi)\partial_x v^{n+1}-(b^n+B)\partial_y\left((b^n+B)\partial_y v^{n+1}+\phi_y Ub^{n+1}\right)\\
\qquad\qquad\qquad-(\phi_{yy}UB+B_x)b^{n+1}+\phi U_x v^{n+1}=r_2,
\end{array}
\right.
\end{align}
with the initial data
\begin{align}
\label{II}
b^{n+1}(0,x,y)=b_{0}(x,y)-B(0,x),\qquad v^{n+1}(0,x,y)=u_0(x,y)-\phi(y) U(0,x),
\end{align}
and the boundary conditions
\begin{align}
\label{BBB}
v^{n+1}|_{y=0}=0,\qquad\lim_{y\rightarrow+\infty}(b^{n+1}, v^{n+1})(t,x,y)=\mathbf0.
\end{align}

Direct computations show that
\begin{align}\label{initial_n+1}
	\partial_t^j {\bf v}^{n+1}(0,x,y)~=~{\bf v}_0^j(x,y)~\in~H^m(\Omega),\quad 0\leq j\leq m.
\end{align}
Then, it is standard to show  the existence of solutions to  the linearized problem (\ref{iteration-n+1})-(\ref{BBB})  in a time interval $t\in [0, T_{n+1}]$ with $0<T_{n+1}\leq T_n$. Moreover,
\begin{align*}
	{\bf v}^{n+1}=(b^{n+1}, v^{n+1})\in \mathcal{H}^m(\Omega_{T_{n+1}}). 
\end{align*}
We only need to derive the uniform estimates of ${\bf v}^{n+1}$ in $\mathcal{H}^m(\Omega_{T_{n+1}})$, which guarantee that the life-span $T_{n+1}$ of the approximate solution ${\bf v}^{n+1}$ has a strictly positive
lower bound as $n$ goes to infinity. In the mean time,  the positive
lower boundedness of $b^{n+1}+B(t, x)$ can be preserved. That is, there is a $0<T_*\leq T$ such that for any $n\geq0,$
\begin{align*}
	T_{n+1}~\geq~T_*,\quad b^{n+1}(t,x,y)+B(t, x)\geq\frac{\delta_0}{2}\quad\mbox{in}~\Omega_{T_{*}}.
\end{align*}

 Now in order to prove the well-posedness of  the nonlinear problem  (\ref{VBE}), (\ref{BCP})-(\ref{FCP}),
 we will show that the functional mapping from ${\bf v}^{n}$ to ${\bf v}^{n+1}$ is a contraction mapping in $L^2$ sense, cf.  \cite{M}. Therefore, it follows that the approximate solution sequence $\{{\bf v}^{n}(t,x,y)\}_{n=0}^{\infty}$ is a Cauchy sequence in $\mathcal{H}^k(\Omega_{T_*}) (k<m)$, and it converges strongly to some function ${\bf v}(t,x,y)=(b, v)(t,x,y)$ in $\mathcal{H}^k(\Omega_{T_*})$. Then, it is straightforward to verify that $(b+B,v+U\phi)$ is a unique solution to (\ref{VBE}), (\ref{BCP})-(\ref{FCP}) by letting $n\rightarrow +\infty$ in (\ref{iteration-n+1}).

For this purpose,  we define the closed set $E_T(M_\ast)$ in $\mathcal{H}^{m}(\Omega_T)$,
\begin{align*}
E_T(M_\ast)=\left\{{\bf f}(t,x,y)=(f_1, f_2)(t,x,y)\in\mathcal{H}^m(\Omega_T): \|{\bf f}\|_{\mathcal{H}^m(\Omega_T)}\leq M_\ast,~f_1+B(t,x)\geq\frac{\delta_0}{2}\right\},
\end{align*}
where $M_\ast$ is a positive constant to be determined later. As mentioned above, suppose ${\bf v}^{n}=(b^n,v^n)\in E_T(M_\ast)$ for some $T$ independent of $n$, we will prove ${\bf v}^{n+1}=(b^{n+1}, v^{n+1})\in E_T(M_\ast)$. That is, we prove the mapping
\begin{align}\label{def-map}
\Pi:\quad {\bf v}^{n}\mapsto {\bf v}^{n+1}
\end{align}
defined by solving the linear iteration problem (\ref{iteration-n+1})-(\ref{BBB}) is a mapping from $E_T(M_\ast)$ to $E_T(M_\ast)$ itself for some small $T$ independent of $n$. Moreover, we also prove the mapping $\Pi$ is a contraction mapping in $L^2$.

Below, we will focus on the uniform energy estimates of ${\bf v}^{n+1}$ defined in (\ref{iteration-n+1}) by assuming that ${\bf v}^{n}\in E_{T_n}$ and the proof is divided into three parts which are given in the next subsections \ref{SUB1} to \ref{SUB3} respectively. Before that, we introduce some notations. Denote by the notation $[\cdot, \cdot]$ the commutator, and by  $\mathcal{P}(\cdot)$   a generic polynomial functional  which may vary from line to line.

\subsection{$L^2$-estimate}\label{SUB1}
In this subsection, we will derive the $L^2$ estimate of solutions to (\ref{iteration-n+1}). Multiplying the first equation in $(\ref{iteration-n+1})$ by $b^{n+1}$ and multiplying the second equation in $(\ref{iteration-n+1})$ by $v^{n+1}$ respectively, adding them together and integrating the resulting equation over $\Omega$, we obtain
\begin{align}\label{est-L2-0}
\begin{aligned}
&\frac12\frac{d}{dt}\int_{\Omega}(b^{n+1})^2+(v^{n+1})^2dxdy\\
&+\int_{\Omega}b^{n+1}\left\{(v^n+U\phi)\partial_x b^{n+1}-(b^n+B)\partial_x v^{n+1}+B_xv^{n+1}-\phi U_x b^{n+1} \right\}dxdy\\
&
+\int_{\Omega}v^{n+1}\left\{-(b^n+B)\partial_xb^{n+1}+(v^n+U\phi)\partial_x v^{n+1}-(b^n+B)\partial_y\left[(b^n+B)\partial_y v^{n+1}+\phi_y U b^{n+1}\right]\right.\\
&\left.\qquad\qquad\quad-(\phi_{yy}UB+B_x)b^{n+1}+\phi U_xv^{n+1}\right\}dxdy\\
&=\int_{\Omega}r_1\cdot b^{n+1}+r_2\cdot v^{n+1}dxdy.
\end{aligned}\end{align}
Note that
\begin{align}\label{est-L2-1}
\begin{aligned}
&\left|\int_{\Omega}b^{n+1}\big[(v^n+U\phi)\partial_x b^{n+1}-(b^n+B)\partial_x v^{n+1} \big]+v^{n+1}\big[-(b^n+B)\partial_x b^{n+1}+(v^n+U\phi)\partial_x v^{n+1}\big]dxdy\right|\\
&=\left|\int_{\Omega}-\partial_x\left(v^n+U\phi\right)\frac{(b^{n+1})^2+(v^{n+1})^2}{2}+\partial_x(b^n+B) b^{n+1}v^{n+1} dxdy\right|\\
&\lesssim \left(1+\|{\bf v}^n(t)\|_{\mathcal{H}^3}\right)\int_{\Omega}|{\bf v}^{n+1}(t)|^2dxdy,
\end{aligned}\end{align}
where the Sobolev embedding inequality is used in the last inequality.

Then, by the boundary condition $(\ref{BBB})$ and integration by parts, we have
\begin{align}\label{est-L2-2}
\begin{aligned}
&-\int_{\Omega}v^{n+1}\cdot(b^n+B)\partial_y\big[(b^n+B)\partial_y v^{n+1}+\phi_y U b^{n+1}\big]dxdy\\
&=\int_{\Omega}\big[(b^n+B)\partial_y v^{n+1}+\partial_y b^n v^{n+1}\big]\cdot\big[(b^n+B)\partial_y v^{n+1}+\phi_y U b^{n+1}\big]dxdy\\
&\geq \frac12\int_{\Omega}\big[(b^n+B)\partial_y v^{n+1}\big]^2dxdy-\frac12\int_{\Omega}(\partial_y b^n v^{n+1})^2+(\phi_y Ub^{n+1})^2dxdy\\
&\geq \frac{\delta_0^2}{8}\int_{\Omega}\big(\partial_y v^{n+1}(t)\big)^2dxdy-C\left(1+\|{\bf v}^n(t)\|_{\mathcal{H}^3}^2\right)\int_{\Omega}|{\bf v}^{n+1}(t)|^2dxdy,
\end{aligned}\end{align}
where the following elementary inequality is used in the  first inequality:
\begin{align}\label{useful-ineq}
(a+b)(a+c)\geq\frac{1}{2}(a^2-b^2-c^2),
\end{align}
and the lower boundedness \eqref{iteration-n} of $b^n+B$   is used  in the second inequality.

Next, it is straightforward to show
\begin{align}\label{est-L2-3}
	&\left|\int_{\Omega}b^{n+1}\cdot(B_xv^{n+1}-\phi U_x b^{n+1})+ v^{n+1}\cdot\big[-(\phi_{yy}UB+B_x)b^{n+1}+\phi U_xv^{n+1}\big]dxdy \right| \notag \\
	&\lesssim \int_{\Omega}|{\bf v}^{n+1}(t)|^2dxdy.
\end{align}
Consequently, plugging \eqref{est-L2-1}-\eqref{est-L2-3} into \eqref{est-L2-0} yields
\begin{align}\label{est-L2}
\begin{aligned}
&\frac{d}{dt}\left\|{\bf v}^{n+1}(t)\right\|_{L^2(\Omega)}^2+\left\|\partial_y v^{n+1}(t)\right\|_{L^2(\Omega)}^2\\
&\lesssim \big(1+\|{\bf v}^n(t)\|^2_{\mathcal{H}^3}\big)\|{\bf v}^{n+1}(t)\|^2_{L^2(\Omega)}+\int_{\Omega}r_1\cdot{b}^{n+1}+r_2\cdot{v}^{n+1}dxdy.
\end{aligned}\end{align}

\subsection{Estimates on tangential derivatives}\label{SUB2}
Applying the tangential derivatives $\partial_\tau^\alpha$ to the first and second equations in $(\ref{iteration-n+1})$ with $|
\alpha|\leq m$ and $m\geq4$, multiplying them by $\partial_\tau^\alpha b^{n+1}$ and $\partial_\tau^\alpha v^{n+1}$ respectively, adding them together and integrating the resulting equation over $\Omega$ yield
\begin{align}\label{est-tan0}
\begin{aligned}	
&\frac12\frac{d}{dt}\int_{\Omega}|\partial_\tau^\alpha{\bf v}^{n+1}|^2(t)dxdy\\
&\quad+\int_{\Omega}\partial_{\tau}^{\alpha}\left\{(v^n+U\phi)\partial_xb^{n+1}-(b^n+B)\partial_xv^{n+1}+B_xv^{n+1}-\phi U_x b^{n+1}\right\}\partial_\tau^\alpha b^{n+1}dxdy\\
&\quad+\int_{\Omega}\partial_{\tau}^{\alpha}\left\{-(b^n+B)\partial_xb^{n+1}+(v^n+U\phi)\partial_xv^{n+1}-(b^n+B)\partial_y\left[(b^n+B)\partial_y v^{n+1}+\phi_y U b^{n+1}\right]\right.\\
&\left.\qquad\qquad\quad-(\phi_{yy}UB+B_x)b^{n+1}+\phi U_xv^{n+1}\right\}
\partial_{\tau}^{\alpha}v^{n+1}dxdy\\
&=\int_{\Omega}\partial_\tau^\alpha r_1\cdot\partial_\tau^\alpha b^{n+1}+\partial_\tau^\alpha r_2\cdot\partial_\tau^\alpha v^{n+1}dxdy .
\end{aligned}\end{align}
Notice that
\begin{align}\label{est-tan2}\begin{aligned}
&\int_{\Omega}\partial_{\tau}^{\alpha}\left((v^n+U\phi)\partial_x b^{n+1}-(b^n+B)\partial_x v^{n+1} \right)\partial_\tau^\alpha b^{n+1}\\
&\qquad+\partial_{\tau}^{\alpha}\left(-(b^n+B)\partial_x b^{n+1}+(v^n+U\phi)\partial_x v^{n+1} \right)\partial_\tau^\alpha v^{n+1}dxdy\\
&=\int_{\Omega}-\partial_x (v^n+U\phi)\frac{|\partial_\tau^\alpha{\bf v}^{n+1} |^2}{2}+\partial_x(b^n+B)\left(\partial_\tau^\alpha b^{n+1}\cdot\partial_\tau^\alpha v^{n+1}\right)dxdy\\
&\quad+\int_{\Omega}\left(\left[\partial_\tau^\alpha, v^n+U\phi\right]\partial_x b^{n+1}-\left[\partial_\tau^\alpha, b^n+B\right]\partial_x v^{n+1}\right)\cdot\partial_{\tau}^{\alpha}b^{n+1}\\
&\qquad\qquad+\left(-\left[\partial_\tau^\alpha, b^n+B\right]\partial_x b^{n+1}+\left[\partial_\tau^\alpha, v^n+U\phi\right]\partial_x v^{n+1}\right)\cdot\partial_{\tau}^{\alpha}v^{n+1}dxdy.
\end{aligned}\end{align}
It follows, from \eqref{Moser} and $m\geq4$, that
\begin{align}\label{est-commutator1}
\begin{aligned}
	\left\|\left[\partial_\tau^\alpha, v^n+U\phi\right]\partial_x b^{n+1} \right\|_{L^2(\Omega)}\lesssim &\sum_{|\beta|=1}\left\|\partial_\tau^\beta(v^n+U\phi)(t)\right\|_{\mathcal{H}^{m-1}}\left\|\partial_xb^{n+1}(t)\right\|_{\mathcal{H}^{m-1}}\\
	\lesssim &(1+\|v^n(t)\|_{\mathcal{H}^m} )\|b^{n+1}(t)\|_{\mathcal{H}^m}.
\end{aligned}\end{align}
By using \eqref{Moser1} and \eqref{OUTE}, we have
\begin{align}\label{est-commutator2}
\begin{aligned}
	\left\|\left[\partial_\tau^\alpha, b^n+B\right]\partial_x v^{n+1} \right\|_{L^2(\Omega)}\lesssim &\left(\sum_{|\beta|=1}\left\|\partial_\tau^\beta b^n(t)\right\|_{\mathcal{H}^{m-1}}+\left\|\partial_\tau^\beta B(t)\right\|_{\mathcal{C}^{m-1}}\right)\left\|\partial_x v^{n+1}(t)\right\|_{\mathcal{H}^{m-1}}\\
	\lesssim &(1+\|b^n(t)\|_{\mathcal{H}^m} )\|v^{n+1}(t)\|_{\mathcal{H}^m}.
\end{aligned}\end{align}
Also,
\begin{align*}
	&\left\|\left[\partial_\tau^\alpha, b^n+B\right]\partial_x b^{n+1} \right\|_{L^2(\Omega)}+\left\|\left[\partial_\tau^\alpha, v^n+U\phi\right]\partial_x v^{n+1} \right\|_{L^2(\Omega)}\lesssim (1+\|{\bf v}^n(t)\|_{\mathcal{H}^m} )\|{\bf v}^{n+1}(t)\|_{\mathcal{H}^m}.
\end{align*}
Consequently, it leads to
\begin{align}\label{est-tan1}
\begin{aligned}
&\Big|\int_{\Omega}\partial_{\tau}^{\alpha}\left((v^n+U\phi)\partial_x b^{n+1}-(b^n+B)\partial_x v^{n+1} \right)\partial_\tau^\alpha b^{n+1}\\
&\quad +\partial_{\tau}^{\alpha}\left(-(b^n+B)\partial_x b^{n+1}+(v^n+U\phi)\partial_x v^{n+1} \right)\partial_\tau^\alpha v^{n+1}dxdy\Big|\\
&\lesssim (1+\|{\bf v}^n(t)\|_{\mathcal{H}^m} )\|{\bf v}^{n+1}(t)\|_{\mathcal{H}^m}^2.
\end{aligned}\end{align}

Next,
\begin{align}\label{est-tan2-0}\begin{aligned}
&-\int_{\Omega}\partial_{\tau}^{\alpha}\left\{(b^n+B)\partial_y\left[(b^n+B)\partial_y v^{n+1}+\phi_y U b^{n+1}\right]\right\}\cdot\partial_{\tau}^{\alpha}v^{n+1}dxdy\\
&=-\int_{\Omega}(b^n+B)\partial_{\tau}^{\alpha}\partial_y\left[(b^n+B)\partial_y v^{n+1}+\phi_y U b^{n+1}\right]\cdot\partial_{\tau}^{\alpha}v^{n+1}dxdy\\
&\quad-\sum_{\beta\leq\alpha, |\beta|=1}\int_{\Omega}C_\alpha^\beta\partial_\tau^\beta (b^n+B)\partial_{\tau}^{\alpha-\beta}\partial_y\left[(b^n+B)\partial_y v^{n+1}+\phi_y U b^{n+1}\right]\cdot\partial_{\tau}^{\alpha}v^{n+1}dxdy\\
&\quad-\sum_{\beta\leq\alpha, |\beta|\geq2}\int_{\Omega}C_\alpha^\beta\partial_\tau^\beta (b^n+B)\partial_{\tau}^{\alpha-\beta}\partial_y\left[(b^n+B)\partial_y v^{n+1}+\phi_y U b^{n+1}\right]\cdot\partial_{\tau}^{\alpha}v^{n+1}dxdy\\
&\triangleq I_1+I_2+I_3.
\end{aligned}\end{align}
By integration by parts and the boundary condition $\partial_\tau^\alpha v^{n+1}|_{y=0}=0$,
\begin{align*}\begin{aligned}
I_1&=\int_{\Omega}\partial_{\tau}^{\alpha}\left[(b^n+B)\partial_y v^{n+1}+\phi_y Ub^{n+1}\right]\cdot\partial_y\left[(b^n+B)\partial_{\tau}^{\alpha}v^{n+1}\right]dxdy\\
&=\int_{\Omega}\left\{(b^n+B)\partial_\tau^\alpha\partial_y v^{n+1}+\left[\partial_{\tau}^{\alpha}, (b^n+B)\right]\partial_y v^{n+1}+\partial_\tau^\alpha(\phi_y Ub^{n+1})\right\}\\
&\qquad\quad\cdot\left[(b^n+B)\partial_{\tau}^{\alpha}\partial_y v^{n+1}+\partial_y b^n \partial_{\tau}^{\alpha} v^{n+1}\right]dxdy\\
&\geq\frac12\int_{\Omega}\left[(b^n+B)\partial_{\tau}^{\alpha}\partial_yv^{n+1}\right]^2dxdy\\
&\quad-\frac12\int_{\Omega}\left\{\left(\left[\partial_\tau^\alpha, (b^n+B)\right]\partial_y v^{n+1}+\partial^\alpha_\tau\left(\phi_y Ub^{n+1}\right) \right)^2+\left(\partial_y b^n \partial_{\tau}^{\alpha} v^{n+1}\right)^2\right\}dxdy,
\end{aligned}\end{align*}
where \eqref{useful-ineq} is used in the  last inequality.

Similar to \eqref{est-commutator1}, we obtain
\begin{align}\label{est-commutator3}
	\left\|\partial^\alpha_\tau\left(\phi_y Ub^{n+1}\right) \right\|_{L^2(\Omega)}\lesssim \|b^{n+1}(t)\|_{\mathcal{H}^m}
\end{align}
due to \eqref{Moser}.

Then combining the above two inequalities together and using  \eqref{iteration-n}, we obtain the following estimate according to the similar argument in \eqref{est-commutator1},
\begin{align}\label{est-tan2-1}\begin{aligned}
I_1\geq & \frac{\delta_0^2}{8}\left\|\partial_{\tau}^{\alpha}\partial_y v^{n+1}\right\|_{L^2(\Omega)}^2-C\left(1+\|b^n(t)\|_{\mathcal{H}^m}^2\right)\|{\bf v}^{n+1}(t)\|_{\mathcal{H}^m}^2.
\end{aligned}\end{align}
Integration by parts yields that
\begin{align*}
I_2= &	\sum_{\beta\leq\alpha, |\beta|=1}\int_{\Omega}C_\alpha^\beta\partial_{\tau}^{\alpha-\beta}\left[(b^n+B)\partial_y v^{n+1}+\phi_y U b^{n+1}\right]\cdot\left[\partial_\tau^\beta (b^n+B)\partial_{\tau}^{\alpha}\partial_y v^{n+1}+\partial_\tau^\beta\partial_y b^n\partial_{\tau}^{\alpha}v^{n+1}\right]dxdy.
\end{align*}
Note that, for $\beta\leq\alpha$ and $|\beta|=1$,
\begin{align*}
	\left\| \partial_{\tau}^{\alpha-\beta}\left[(b^n+B)\partial_y v^{n+1}+\phi_y U b^{n+1}\right]\right\|_{L^2(\Omega)}
	&\lesssim\left(1+\|b^n(t)\|_{\mathcal{H}^{m-1}}\right)\|\partial_y v^{n+1} (t)\|_{\mathcal{H}^{m-1}}+\|b^{n+1} (t)\|_{\mathcal{H}^{m-1}}\\
	&\lesssim\left(1+\|b^n(t)\|_{\mathcal{H}^{m-1}}\right)\|{\bf v}^{n+1} (t)\|_{\mathcal{H}^{m}}.
\end{align*}
Therefore, it holds
\begin{align}\label{est-tan2-2}
	\begin{aligned}
		|I_2|\lesssim &\sum_{\beta\leq\alpha, |\beta|=1}\left(1+\|b^n(t)\|_{\mathcal{H}^{m-1}}\right)\|{\bf v}^{n+1} (t)\|_{\mathcal{H}^{m}}\cdot\left(\|\partial_\tau^\beta (b^n+B)\|_{L^\infty(\Omega)}\|\partial_{\tau}^{\alpha}\partial_y v^{n+1}\|_{L^2(\Omega)}\right.\\
		&\qquad\qquad\qquad\left.+\|\partial_\tau^\beta\partial_y b^n\|_{L^\infty}\|\partial_{\tau}^{\alpha}v^{n+1}\|_{L^2(\Omega)} \right)\\
		\leq &\frac{\delta_0^2}{16}\left\|\partial_{\tau}^{\alpha}\partial_y v^{n+1}\right\|_{L^2(\Omega)}^2+C\left(1+\|b^n(t)\|_{\mathcal{H}^{m}}^4\right)\|{\bf v}^{n+1} (t)\|_{\mathcal{H}^{m}}^2
	\end{aligned}
\end{align}
provided $m\geq4$.

For $I_3$, since $\beta\leq\alpha$ and $|\beta|\geq2$, we write
\begin{align*}
	&\partial_\tau^\beta (b^n+B) \partial_{\tau}^{\alpha-\beta}\partial_y\left[(b^n+B)\partial_y v^{n+1}+\phi_y U b^{n+1}\right]\\
	&=\partial_\tau^{\beta-\gamma}(\partial_\tau^\gamma b^n+\partial_\tau^\gamma B )\cdot \partial_{\tau}^{\alpha-\beta}\partial_y\left[(b^n+B)\partial_y v^{n+1}+\phi_y U b^{n+1}\right],
\end{align*}
for some $\gamma\leq\beta, |\gamma|=2$, and then
\begin{align*}
	&\left\|\partial_\tau^\beta (b^n+B) \partial_{\tau}^{\alpha-\beta}\partial_y\left[(b^n+B)\partial_y v^{n+1}+\phi_y U b^{n+1}\right]\right\|_{L^2(\Omega)}\\
	&\lesssim\left(\|\partial_\tau^\gamma b^n(t)\|_{\mathcal{H}^{m-2}}+\|\partial_\tau^\gamma B(t)\|_{\mathcal{C}^{m-2}}\right)\cdot \left\|\partial_y\left[(b^n+B)\partial_y v^{n+1}+\phi_y U b^{n+1}\right](t)\right\|_{\mathcal{H}^{m-2}}\\
	&\lesssim\left(1+\|b^n(t)\|_{\mathcal{H}^{m}}\right)\cdot\left(\|(b^n+B)\partial_y v^{n+1}\|_{\mathcal{H}^{m-1}}+\|\phi_y U b^{n+1}\|_{\mathcal{H}^{m-1}}\right)\\
	&\lesssim \left(1+\|b^n(t)\|_{\mathcal{H}^{m}}^2\right)\|{\bf v}^{n+1} (t)\|_{\mathcal{H}^{m}},
\end{align*}
provided $m\geq4$.  As a consequence,
\begin{align}\label{est-tan2-3}
	\begin{aligned}
		|I_3|\lesssim \left(1+\|b^n(t)\|_{\mathcal{H}^{m}}^2\right)\|{\bf v}^{n+1} (t)\|_{\mathcal{H}^{m}}\cdot\|\partial_{\tau}^{\alpha}v^{n+1}\|_{L^2(\Omega)}\lesssim \left(1+\|b^n(t)\|_{\mathcal{H}^{m}}^2\right)\|{\bf v}^{n+1} (t)\|_{\mathcal{H}^{m}}^2.
			\end{aligned}
\end{align}
Substituting \eqref{est-tan2-1}, \eqref{est-tan2-2} and \eqref{est-tan2-3} into \eqref{est-tan2-0}, we arrive at
\begin{align}\label{est-tan2}\begin{aligned}
&-\int_{\Omega}\partial_{\tau}^{\alpha}\left\{(b^n+B)\partial_y\left[(b^n+B)\partial_y v^{n+1}+\phi_y U b^{n+1}\right]\right\}\cdot\partial_{\tau}^{\alpha}v^{n+1}dxdy\\
&\geq \frac{\delta_0^2}{16}\left\|\partial_{\tau}^{\alpha}\partial_y v^{n+1}\right\|_{L^2(\Omega)}^2-C\left(1+\|b^n(t)\|_{\mathcal{H}^m}^4\right)\|{\bf v}^{n+1}(t)\|_{\mathcal{H}^m}^2.
\end{aligned}\end{align}

We now turn to the estimates of the other terms in \eqref{est-tan0}.
By using \eqref{Moser} and \eqref{Moser1}, it follows that
\begin{align}\label{est-commutator4}
\begin{aligned}
	\left\|\partial_\tau^\alpha(B_xv^{n+1}-\phi U_x b^{n+1})\right\|_{L^2(\Omega)}&\lesssim \|{\bf v}^{n+1}(t)\|_{\mathcal{H}^m},\\
	\left\|\partial_\tau^\alpha\left[-(\phi_{yy}UB+B_x)b^{n+1}+\phi U_xv^{n+1}\right]\right\|_{L^2(\Omega)}&\lesssim \|{\bf v}^{n+1}(t)\|_{\mathcal{H}^m},
\end{aligned}\end{align}
then,
\begin{align}\label{est-tan3}
\begin{aligned}
&\left|\int_{\Omega}\partial_\tau^\alpha(B_xv^{n+1}-\phi U_x b^{n+1})\cdot\partial_\tau^\alpha b^{n+1}+ \partial_\tau^\alpha\left[-(\phi_{yy}UB+B_x)b^{n+1}+\phi U_xv^{n+1}\right]\cdot\partial_\tau^\alpha v^{n+1} dxdy\right|\\
&\lesssim \|{\bf v}^{n+1}(t)\|_{\mathcal{H}^m}^2.	
\end{aligned}\end{align}
Substituting \eqref{est-tan1}, \eqref{est-tan2} and \eqref{est-tan3} into \eqref{est-tan0} yields
\begin{align}\label{E3}
\begin{aligned}
&\frac{d}{dt}\left\|\partial_\tau^\alpha {\bf v}^{n+1}(t)\right\|_{L^2(\Omega)}^2+\left\|\partial_\tau^\alpha\partial_y v^{n+1}(t)\right\|_{L^2(\Omega)}^2\\
&\lesssim \big(1+\|{\bf v}^n(t)\|_{\mathcal{H}^m}^4\big)\|{\bf v}^{n+1}(t)\|^2_{\mathcal{H}^m}+\|\partial_\tau^\alpha r_1(t)\|^2_{L^2(\Omega)}+\|\partial_\tau^\alpha r_2(t)\|^2_{L^2(\Omega)}.
\end{aligned}\end{align}
Hence, summing the above estimates over $\alpha, |\alpha|\leq m$ gives
\begin{align}\label{est-tan}
\begin{aligned}
&\frac{d}{dt}\left\|{\bf v}^{n+1}(t)\right\|_{\mathcal{B} ^{m,0}}^2+\left\|\partial_y v^{n+1}(t)\right\|_{\mathcal{B} ^{m,0}}^2\lesssim \big(1+\|{\bf v}^n(t)\|_{\mathcal{H}^m}^4\big)\|{\bf v}^{n+1}(t)\|^2_{\mathcal{H}^m}+\|r_1(t)\|^2_{\mathcal{B} ^{m,0}}+\| r_2(t)\|^2_{\mathcal{B} ^{m,0}}.
\end{aligned}\end{align}

\subsection{Estimates on normal derivatives}
\label{SUB3}
In this subsection, we will derive the estimates on the normal derivatives in the subsequent three steps.

{\bf Step I.} Applying the operator $\partial_\tau^{\alpha} (|\alpha|\leq m-1)$ to the second equation in $(\ref{iteration-n+1})$, multiplying the resulting equation by $\partial_\tau^{\alpha}\partial_tv^{n+1}$  and integrating it over $\Omega$ lead to
\begin{align}\label{est-v_t0}
\begin{aligned}
&\int_{\Omega}(\partial_{\tau}^{\alpha}\partial_t v^{n+1})^2dxdy +\int_{\Omega}\partial_{\tau}^{\alpha}\left\{-(b^n+B)\partial_xb^{n+1}+(v^n+U\phi)\partial_xv^{n+1}-(\phi_{yy}UB+B_x)b^{n+1}+\phi U_xv^{n+1}\right.\\
&\left.\qquad\qquad-(b^n+B)\partial_y\left[(b^n+B)\partial_y v^{n+1}+\phi_y U b^{n+1}\right]\right\}
\partial_{\tau}^{\alpha}\partial_t v^{n+1}dxdy=\int_{\Omega}\partial_\tau^\alpha  r_2\cdot\partial_\tau^\alpha\partial_t v^{n+1}dxdy .
\end{aligned}\end{align}
Since $|\alpha|\leq m-1$ and $m\geq4$, by the similar arguments to those of \eqref{est-commutator1} and \eqref{est-commutator2}, one has
\begin{align*}
	&\left\|\partial_{\tau}^{\alpha}\left[-(b^n+B)\partial_xb^{n+1}+(v^n+U\phi)\partial_xv^{n+1}-(\phi_{yy}UB+B_x)b^{n+1}+\phi U_xv^{n+1}\right] \right\|_{L^2(\Omega)}\\
	&\lesssim \left(1+\|{\bf v}^n(t)\|_{\mathcal{H}^{m-1}}\right)\|\partial_x{\bf v}^{n+1}(t)\|_{\mathcal{H}^{m-1}}+\|{\bf v}^{n+1}(t)\|_{\mathcal{H}^{m-1}}\\
	&\lesssim \left(1+\|{\bf v}^n(t)\|_{\mathcal{H}^{m-1}}\right)\|{\bf v}^{n+1}(t)\|_{\mathcal{H}^{m}}.
\end{align*}
Then
\begin{align}\label{est-v_t1}	
\begin{aligned}
&\left|\int_{\Omega}\partial_{\tau}^{\alpha}\left[-(b^n+B)\partial_xb^{n+1}+(v^n+U\phi)\partial_xv^{n+1}-(\phi_{yy}UB+B_x)b^{n+1}+\phi U_xv^{n+1}\right] \cdot\partial_{\tau}^{\alpha}\partial_tv^{n+1}dxdy\right|\\
&\leq\frac{1}{6} \|\partial_\tau^\alpha\partial_t v^{n+1}(t)\|_{L^2(\Omega)}^2 +C \left(1+\|{\bf v}^n(t)\|_{\mathcal{H}^{m-1}}^2\right)\|{\bf v}^{n+1}(t)\|_{\mathcal{H}^{m}}^2.
\end{aligned}\end{align}

Similarly to  \eqref{est-tan2-0}, we have
\begin{align}\label{est-v_t2-0}\begin{aligned}
&-\int_{\Omega}\partial_{\tau}^{\alpha}\left\{(b^n+B)\partial_y\left[(b^n+B)\partial_y v^{n+1}+\phi_y U b^{n+1}\right]\right\}\cdot\partial_{\tau}^{\alpha}\partial_t v^{n+1}dxdy\\
&=-\int_{\Omega}(b^n+B)\partial_{\tau}^{\alpha}\partial_y\left[(b^n+B)\partial_y v^{n+1}+\phi_y U b^{n+1}\right]\cdot\partial_{\tau}^{\alpha}\partial_t v^{n+1}dxdy\\
&\quad-\sum_{\beta\leq\alpha, |\beta|\geq1}\int_{\Omega}C_\alpha^\beta\partial_\tau^\beta (b^n+B)\partial_{\tau}^{\alpha-\beta}\partial_y\left[(b^n+B)\partial_y v^{n+1}+\phi_y U b^{n+1}\right]\cdot\partial_{\tau}^{\alpha}\partial_t v^{n+1}dxdy\\
&=
J_1+J_2.
\end{aligned}\end{align}
By integration by parts and the boundary condition $\partial_\tau^\alpha\partial_t v^{n+1}|_{y=0}=0$,
\begin{align}\label{est-J_1-0}
\begin{aligned}
J_1&=\int_{\Omega}\partial_{\tau}^{\alpha}\left[(b^n+B)\partial_y v^{n+1}+\phi_y Ub^{n+1}\right]\cdot\left[(b^n+B)\partial_y \partial_{\tau}^{\alpha}\partial_t v^{n+1}+\partial_y b^n \partial_{\tau}^{\alpha}\partial_t v^{n+1}\right]dxdy\\
&=\int_{\Omega}\partial_{\tau}^{\alpha}\left[(b^n+B)\partial_y v^{n+1}+\phi_y Ub^{n+1}\right]\cdot(b^n+B)\partial_t \partial_{\tau}^{\alpha}\partial_y v^{n+1}dxdy\\
&\quad+\int_{\Omega}\partial_{\tau}^{\alpha}\left[(b^n+B)\partial_y v^{n+1}+\phi_y Ub^{n+1}\right]\cdot \partial_y b^n \partial_{\tau}^{\alpha}\partial_t v^{n+1}dxdy\\
&
=J_1^1+J_1^2.
\end{aligned}\end{align}
Here
\begin{align*}\begin{aligned}
J_1^1&=\int_{\Omega}\left\{(b^n+B)\partial_\tau^\alpha\partial_y v^{n+1}+\left[\partial_{\tau}^{\alpha}, (b^n+B)\right]\partial_y v^{n+1}+\partial_\tau^\alpha(\phi_y Ub^{n+1})\right\}\cdot (b^n+B)\partial_t \partial_{\tau}^{\alpha}\partial_y v^{n+1} dxdy\\
&=\frac12\frac{d}{dt}\int_{\Omega}\left[(b^n+B)\partial_{\tau}^{\alpha}\partial_y v^{n+1}\right]^2dxdy-\frac12\int_{\Omega}\partial_t(b^n+B)^2\left(\partial_{\tau}^{\alpha}\partial_y v^{n+1}\right)^2dxdy\\
&\quad+\int_{\Omega}(b^n+B)\partial_y \partial_{\tau}^{\alpha}\partial_t v^{n+1}\cdot\left\{\left[\partial_\tau^\alpha, (b^n+B)\right]\partial_y v^{n+1}+\partial^\alpha_\tau(\phi_y Ub^{n+1}) \right\}dxdy\\
&\geq \frac12\frac{d}{dt}\int_{\Omega}\left[(b^n+B)\partial_{\tau}^{\alpha}\partial_y v^{n+1}\right]^2dxdy-\frac12\|\partial_t(b^n+B)^2\|_{L^\infty(\Omega)}\left\|\partial_{\tau}^{\alpha}\partial_y v^{n+1}\right\|_{L^2(\Omega)}^2\\
&\quad-\|b^n+B\|_{L^\infty(\Omega)}\left\|\partial_y\partial_{\tau}^{\alpha}\partial_t v^{n+1}\right\|_{L^2(\Omega)}\cdot\left\| \left[\partial_\tau^\alpha, (b^n+B)\right]\partial_y v^{n+1}+\partial^\alpha_\tau(\phi_y Ub^{n+1})\right\|_{L^2(\Omega)}.
\end{aligned}\end{align*}
Similar to \eqref{est-commutator2} and \eqref{est-commutator3}, it holds, for $|\alpha|\leq m-1$, that
\begin{align*}
	\left\| \left[\partial_\tau^\alpha, (b^n+B)\right]\partial_y v^{n+1}+\partial^\alpha_\tau(\phi_y Ub^{n+1})\right\|_{L^2(\Omega)}&\lesssim (1+\|b^n(t)\|_{\mathcal{H}^{m-1}} )\|\partial_y v^{n+1}(t)\|_{\mathcal{H}^{m-2}}+\|b^{n+1}(t)\|_{\mathcal{H}^{m-1}}\\
	&\lesssim (1+\|b^n(t)\|_{\mathcal{H}^{m-1}} )\|{\bf v}^{n+1}(t)\|_{\mathcal{H}^{m-1}}.
\end{align*}
Consequently, for $|\alpha|\leq m-1$, it follows that
\begin{align}\label{est-J_1-1}\begin{aligned}
J_1^1&\geq \frac12\frac{d}{dt}\int_{\Omega}\left[(b^n+B)\partial_{\tau}^{\alpha}\partial_y v^{n+1}\right]^2dxdy-\frac{1}{8}\left\|\partial_y\partial_{\tau}^{\alpha}\partial_t v^{n+1}\right\|_{L^2(\Omega)}^2\\
&\quad-C (1+\|b^n(t)\|_{\mathcal{H}^{m-1}}^4 )\|{\bf v}^{n+1}(t)\|_{\mathcal{H}^m}^2.
\end{aligned}\end{align}
For $J_1^2$, notice that, for $|\alpha|\leq m-1$,
\begin{align*}
	\left\| \partial_{\tau}^{\alpha}\left[(b^n+B)\partial_y v^{n+1}+\phi_y Ub^{n+1}\right]\right\|_{L^2(\Omega)}&\lesssim (1+\|b^n(t)\|_{\mathcal{H}^{m-1}} )\|\partial_y v^{n+1}(t)\|_{\mathcal{H}^{m-1}}+\|b^{n+1}(t)\|_{\mathcal{H}^{m-1}}\\
	&\lesssim (1+\|b^n(t)\|_{\mathcal{H}^{m-1}} )\|{\bf v}^{n+1}(t)\|_{\mathcal{H}^{m}},
\end{align*}
which implies that
\begin{align}\label{est-J_1-2}
\begin{aligned}
	|J_1^2|\lesssim &(1+\|b^n(t)\|_{\mathcal{H}^{m-1}} )\|{\bf v}^{n+1}(t)\|_{\mathcal{H}^{m}}\cdot \|\partial_y b^n(t)\|_{L^\infty(\Omega)}\|\partial_\tau^\alpha\partial_t v^{n+1}(t)\|_{L^2(\Omega)}\\
	\leq& \frac{1}{6} \|\partial_\tau^\alpha\partial_t v^{n+1}(t)\|_{L^2(\Omega)}^2 +C(1+\|b^n(t)\|_{\mathcal{H}^{m-1}}^4 )\|{\bf v}^{n+1}(t)\|_{\mathcal{H}^{m}}^2,
\end{aligned}\end{align}
if $m\geq4$. Therefore, substituting \eqref{est-J_1-1} and \eqref{est-J_1-2} into \eqref{est-J_1-0} gives
\begin{align}\label{est-J_1}\begin{aligned}
J_1 &\geq \frac12\frac{d}{dt}\int_{\Omega}\left[(b^n+B)\partial_{\tau}^{\alpha}\partial_y v^{n+1}\right]^2dxdy-\frac{1}{8}\left\|\partial_y\partial_{\tau}^{\alpha}\partial_t v^{n+1}\right\|_{L^2(\Omega)}^2-\frac{1}{6} \|\partial_\tau^\alpha\partial_t v^{n+1}(t)\|_{L^2(\Omega)}^2\\
&\quad-C (1+\|b^n(t)\|_{\mathcal{H}^{m-1}}^4 )\|{\bf v}^{n+1}(t)\|_{\mathcal{H}^m}^2.
\end{aligned}\end{align}
$J_2$  can be estimated as  $I_3$ in \eqref{est-tan2-0}. Thus,  for $|\alpha|\leq m-1$ with $m\geq4$,  we obtain that
\begin{align*}
	&\left\|\partial_\tau^\beta (b^n+B) \partial_{\tau}^{\alpha-\beta}\partial_y\left[(b^n+B)\partial_y v^{n+1}+\phi_y U b^{n+1}\right]\right\|_{L^2(\Omega)}\lesssim \left(1+\|b^n(t)\|_{\mathcal{H}^{m-1}}^2\right)\|{\bf v}^{n+1} (t)\|_{\mathcal{H}^{m}}.
\end{align*}
Hence
\begin{align}\label{est-J_2}
	\begin{aligned}
		|J_2|\lesssim &\left(1+\|b^n(t)\|_{\mathcal{H}^{m-1}}^2\right)\|{\bf v}^{n+1} (t)\|_{\mathcal{H}^{m}}\cdot\|\partial_{\tau}^{\alpha}\partial_t v^{n+1}\|_{L^2(\Omega)}\\
		\leq &\frac{1}{6} \|\partial_\tau^\alpha\partial_t v^{n+1}(t)\|_{L^2(\Omega)}^2 +C \left(1+\|b^n(t)\|_{\mathcal{H}^{m-1}}^4\right)\|{\bf v}^{n+1} (t)\|_{\mathcal{H}^{m}}^2.
			\end{aligned}
\end{align}
Substituting \eqref{est-J_1} and \eqref{est-J_2} into \eqref{est-v_t2-0} leads to 
\begin{align}\label{est-v_t2}\begin{aligned}
&-\int_{\Omega}\partial_{\tau}^{\alpha}\left\{(b^n+B)\partial_y\left[(b^n+B)\partial_y v^{n+1}+\phi_y U b^{n+1}\right]\right\}\cdot\partial_{\tau}^{\alpha}\partial_t v^{n+1}dxdy\\
&\geq \frac12\frac{d}{dt}\int_{\Omega}\left[(b^n+B)\partial_{\tau}^{\alpha}\partial_y v^{n+1}\right]^2dxdy-\frac{1}{8}\left\|\partial_y\partial_{\tau}^{\alpha}\partial_t v^{n+1}\right\|_{L^2(\Omega)}^2-\frac{1}{3} \|\partial_\tau^\alpha\partial_t v^{n+1}(t)\|_{L^2(\Omega)}^2\\
&\quad -C (1+\|b^n(t)\|_{\mathcal{H}^{m-1}}^4 )\|{\bf v}^{n+1}(t)\|_{\mathcal{H}^m}^2.
\end{aligned}\end{align}
And
\begin{align}\label{est-v_t3}
	\left|\int_{\Omega}\partial_\tau^\alpha  r_2\cdot\partial_\tau^\alpha\partial_t v^{n+1}dxdy \right|\leq\frac{1}{6} \|\partial_\tau^\alpha\partial_t v^{n+1}(t)\|_{L^2(\Omega)}^2+\frac{3}{2} \|\partial_\tau^\alpha r_2(t)\|_{L^2(\Omega)}^2.
\end{align}
Combining \eqref{est-v_t0}, \eqref{est-v_t1}, \eqref{est-v_t2} and  \eqref{est-v_t3} together, we obtain
\begin{align}\label{est-v_t4}
\begin{aligned}
	&\frac{d}{dt}\int_{\Omega}\left[(b^n+B)\partial_{\tau}^{\alpha}\partial_y v^{n+1}\right]^2dxdy+\|\partial_\tau^\alpha\partial_t v^{n+1}(t)\|_{L^2(\Omega)}^2\\
	&\leq \frac{1}{4}\left\|\partial_y\partial_{\tau}^{\alpha}\partial_t v^{n+1}(t)\right\|_{L^2(\Omega)}^2+C\|\partial_\tau^\alpha r_2(t)\|_{L^2(\Omega)}^2+C(1+\|{\bf v}^n(t)\|_{\mathcal{H}^{m-1}}^4 )\|{\bf v}^{n+1}(t)\|_{\mathcal{H}^m}^2.
\end{aligned}\end{align}
Summing it over $\alpha$ for $|\alpha|\leq m-1$ yields
\begin{align}\label{est-v_y}
	\begin{aligned}
		&\frac{d}{dt}\int_{\Omega}\sum_{|\alpha|\leq m-1}\left[(b^n+B)\partial_{\tau}^{\alpha}\partial_y v^{n+1}\right]^2dxdy+\|\partial_t v^{n+1}(t)\|_{\mathcal{B}^{m-1,0}}^2\\
	&\leq \frac{1}{4}\left\|\partial_y v^{n+1}(t)\right\|_{\mathcal{B}^{m,0}}^2+C\| r_2(t)\|_{\mathcal{B}^{m-1,0}}^2+C(1+\|{\bf v}^n(t)\|_{\mathcal{H}^{m-1}}^4 )\|{\bf v}^{n+1}(t)\|_{\mathcal{H}^m}^2.
	\end{aligned}
\end{align}

{\bf Step II.} In this step, we will show the following estimates for any $t\in[0,T_{n+1})$ and $m\geq4$:
\begin{align}\label{est-v_yy1}
	\|\partial_y^2 v^{n+1}(t)\|_{\mathcal{H}^{m-1}}\lesssim1+
	\left(1+\|{\bf v}^n(t)\|_{\mathcal{H}^{m}}^2\right)\| {\bf v}^{n+1}(t)\|_{\mathcal{H}^{m}},
\end{align}
and
\begin{align} \label{est-v_yy2}
\|\partial_y^2 v^{n+1}(t)\|_{\mathcal{H}^{m-2}}\lesssim1+
	\int_0^t\left(1+\|{\bf v}^n(s)\|_{\mathcal{H}^{m}}^2\right)\| {\bf v}^{n+1}(s)\|_{\mathcal{H}^{m}}ds.
\end{align}	
In fact, we only need to prove  \eqref{est-v_yy1}. Because for any $\alpha\in\mathbb{N}^3$ with $|\alpha|\leq m-2,$
\begin{align*}
	\|\partial^\alpha\partial_y^2 v^{n+1}(t)\|_{L^2(\Omega)}\leq &\|\partial^\alpha\partial_y^2 v^{n+1}(0)\|_{L^2(\Omega)}+\int_0^t \|\partial_t\partial^\alpha\partial_y^2 v^{n+1}(s)\|_{L^2(\Omega)}ds\\
	\leq &\|\partial^\alpha\partial_y^2 v^{n+1}(0)\|_{L^2(\Omega)}+\int_0^t \|\partial_y^2 v^{n+1}(s)\|_{\mathcal{H}^{m-1}}ds,
\end{align*}
then \eqref{est-v_yy2} follows from \eqref{regular-t} and \eqref{est-v_yy1} immediately.

By the second equation in $(\ref{iteration-n+1})$, we write
\begin{align}\label{N2}
\begin{aligned}
(b^n+B)^2\partial_y^2v^{n+1}=&\partial_t v^{n+1}-(b^n+B)\partial_x b^{n+1}+(v^n+U\phi)\partial_x v^{n+1}-(b^n+B)\partial_y b^n\partial_y v^{n+1}\\
&-(b^n+B)\partial_y\big(\phi_y Ub^{n+1}\big)-(\phi_{yy} UB +B_x) b^{n+1}+\phi U_x v^{n+1}-r_2.
\end{aligned}\end{align}
To show \eqref{est-v_yy1}, we apply the operator $\partial^{\alpha} (|\alpha|\leq m-1)$ on $(\ref{N2})$ to obtain
\begin{align*}
&(b^n+B)^2\partial^\alpha\partial_y^2v^{n+1}+\left[\partial^\alpha,  (b^n+B)^2\right]\partial_y^2v^{n+1}\\
=&\partial^\alpha\partial_t v^{n+1}+\partial^\alpha\left\{-(b^n+B)\partial_x b^{n+1}+(v^n+U\phi)\partial_x v^{n+1}
-(b^n+B)\partial_y b^n\partial_y v^{n+1}\right.\\
&\qquad\qquad\qquad\quad\left.-(b^n+B)\partial_y\big(\phi_y Ub^{n+1}\big)-(\phi_{yy}UB+B_x) b^{n+1}+\phi U_x v^{n+1}\right\}-\partial^\alpha r_2.
\end{align*}
Then,
\begin{align*}
	\left\|(b^n+B)^2\partial^\alpha\partial_y^2 v^{n+1}\right\|_{L^2(\Omega)}
	\leq &\left\|\left[\partial^\alpha,  (b^n+B)^2\right]\partial_y^2v^{n+1}\right\|_{L^2(\Omega)}+\left\|\partial^\alpha \partial_t v^{n+1}\right\|_{L^2(\Omega)}+\left\|\partial^\alpha r_2\right\|_{L^2(\Omega)}\\
	&+\left\|\partial^\alpha\left\{-(b^n+B)\partial_x b^{n+1}+(v^n+U\phi)\partial_x v^{n+1}-(b^n+B)\partial_y b^n\partial_y v^{n+1}\right.\right.\\
&\qquad\quad\left.\left.-(b^n+B)\partial_y\big(\phi_y Ub^{n+1}\big)-(\phi_{yy}UB+B_x) b^{n+1}+\phi U_x v^{n+1}\right\}\right\|_{L^2(\Omega)}.
\end{align*}
As $|\alpha|\leq m-1$ and $m\geq4$, similar to \eqref{est-commutator1} and \eqref{est-commutator2}, one has
\begin{align*}
\begin{aligned}
	\left\|\left[\partial^\alpha,  (b^n+B)^2\right]\partial_y^2v^{n+1}\right\|_{L^2(\Omega)}\lesssim &\left(1+\sum_{\beta\leq\alpha,|\beta|=1}\|\partial^\beta b^n(t)\|_{\mathcal{H}^{m-2}}^2\right)\|\partial_y^2 v^{n+1}(t)\|_{\mathcal{H}^{m-2}}\\
	\lesssim &\left(1+\| b^n(t)\|_{\mathcal{H}^{m-1}}^2\right)\| v^{n+1}(t)\|_{\mathcal{H}^{m}},
\end{aligned}\end{align*}
and
\begin{align*}
\begin{aligned}
	&\left\|\partial^\alpha\left\{-(b^n+B)\partial_x b^{n+1}+(v^n+U\phi)\partial_x v^{n+1}-(b^n+B)\partial_y b^n\partial_y v^{n+1}\right.\right.\\
&\qquad\left.\left.-(b^n+B)\partial_y\big(\phi_y Ub^{n+1}\big)-(\phi_{yy}UB+B_x) b^{n+1}+\phi U_x v^{n+1}\right\}\right\|_{L^2(\Omega)}\\
&\lesssim\Big[1+\|{\bf v}^n(t)\|_{\mathcal{H}^{m-1}}+(1+\|b^n(t)\|_{\mathcal{H}^{m-1}})\|\partial_y b^n(t)\|_{\mathcal{H}^{m-1}} \Big]\\
&\quad\cdot\Big(\|\partial_x{\bf v}^{n+1}(t)\|_{\mathcal{H}^{m-1}}+\|\partial_y{\bf v}^{n+1}(t)\|_{\mathcal{H}^{m-1}}+\|{\bf v}^{n+1}(t)\|_{\mathcal{H}^{m-1}} \Big)\\
&\lesssim\left(1+\| {\bf v}^n(t)\|_{\mathcal{H}^{m}}^2\right)\| {\bf v}^{n+1}(t)\|_{\mathcal{H}^{m}}.
\end{aligned}\end{align*}
Thus combining the above three estimates and using \eqref{iteration-n}, we have
\begin{align*}
	\left\|\partial^\alpha\partial_y^2 v^{n+1}(t)\right\|_{L^2(\Omega)}
\lesssim &\left\|\partial^\alpha r_2(t)\right\|_{L^2(\Omega)}+\left(1+\|{\bf v}^n(t)\|_{\mathcal{H}^{m}}^2\right)\|{\bf v}^{n+1}(t)\|_{\mathcal{H}^{m}}.
\end{align*}
It follows that
\begin{align}\label{est-v_yy3}
	\left\|\partial_y^2 v^{n+1}(t)\right\|_{\mathcal{H}^{m-1}}
\lesssim &\left\| r_2(t)\right\|_{\mathcal{H}^{m-1}}+\left(1+\|{\bf v}^n(t)\|_{\mathcal{H}^{m}}^2\right)\|{\bf v}^{n+1}(t)\|_{\mathcal{H}^{m}}.
\end{align}
This implies \eqref{est-v_yy1} by using \eqref{est-r}.

Moreover, combining \eqref{est-tan} with \eqref{est-v_yy3} gives
\begin{align}\label{est-tan-y}
\begin{aligned}
&\frac{d}{dt}\left\|{\bf v}^{n+1}(t)\right\|_{\mathcal{B} ^{m,0}}^2+\left\|\partial_y v^{n+1}(t)\right\|_{\mathcal{H} ^{m}}^2\lesssim \big(1+\|{\bf v}^n(t)\|_{\mathcal{H}^m}^4\big)\|{\bf v}^{n+1}(t)\|^2_{\mathcal{H}^m}+\|r_1(t)\|^2_{\mathcal{B} ^{m,0}}+\| r_2(t)\|^2_{\mathcal{H} ^{m}}.
\end{aligned}\end{align}

{\bf Step III.}  Applying the operator $\partial^{\alpha}=\partial_t^{\alpha_0}\partial_x^{\alpha_1} \partial_y^{\alpha_2}$ to the first equation in $(\ref{iteration-n+1})$ with $ |\alpha|\leq m$ and $\alpha_2\geq1$,
then multiplying the resulting equation by $\partial^{\alpha} b^{n+1}$ and integrating it over $\Omega$
yield
\begin{align}\label{est-b_y0}
\begin{aligned}
&\frac12\frac{d}{dt}\int_{\Omega}(\partial^{\alpha} b^{n+1})^2dxdy+\int_{\Omega}\partial^{\alpha}\left((v^n+U\phi)\partial_xb^{n+1}\right)\cdot\partial^\alpha b^{n+1} dxdy\\
&-\int_{\Omega}\partial^\alpha\left( (b^n+B)\partial_x v^{n+1}\right)\cdot\partial^\alpha b^{n+1} dxdy +\int_{\Omega}\partial^{\alpha}\left(B_x v^{n+1}-\phi U_x b^{n+1}\right)\cdot\partial^\alpha b^{n+1} dxdy\\
&=\int_{\Omega}\partial^\alpha r_1\cdot\partial^\alpha b^{n+1}dxdy .
\end{aligned}\end{align}
For the second term on the left-hand side of the above equality, we have
\begin{align*}
\begin{aligned}
	&\int_{\Omega}\partial^{\alpha}\left((v^n+U\phi)\partial_x b^{n+1}\right)\cdot\partial^\alpha b^{n+1} dxdy\\
	&= \int_{\Omega}\left((v^n+U\phi)\partial_x \partial^{\alpha} b^{n+1}+\left[\partial^\alpha, (v^n+U\phi)\right]\partial_x b^{n+1}\right)\cdot\partial^\alpha b^{n+1} dxdy\\
	&=-\frac{1}{2}\int_{\Omega}\partial_x(v^n+U\phi)\cdot\left(\partial^{\alpha} b^{n+1}\right)^2dxdy+\int_{\Omega}\left[\partial^\alpha, (v^n+U\phi)\right]\partial_x b^{n+1}\cdot\partial^\alpha b^{n+1}dxdy.
\end{aligned}\end{align*}
And for $|\alpha|\leq m$ and $m\geq4$, similar to \eqref{est-commutator1}, it holds that
\begin{align}\label{app-Moser}
\begin{aligned}
	\left\|\left[\partial^\alpha, (v^n+U\phi)\right]\partial_x b^{n+1}\right\|_{L^2(\Omega)}\lesssim &\sum_{\beta\leq\alpha, |\beta|=1}\left\|\partial^\beta(v^n+U\phi)(t)\right\|_{\mathcal{H}^{m-1}}\left\|\partial_xb^{n+1}(t)\right\|_{\mathcal{H}^{m-1}}\\
	\lesssim &(1+\|v^n(t)\|_{\mathcal{H}^m} )\|b^{n+1}(t)\|_{\mathcal{H}^m}
\end{aligned}\end{align}
due to \eqref{Moser}. Then
\begin{align}\label{est-b_y1}
\begin{aligned}
	&\left|\int_{\Omega}\partial^{\alpha}\left((v^n+U\phi)\partial_x b^{n+1}\right)\cdot\partial^\alpha b^{n+1} dxdy\right|
	\lesssim (1+\|v^n(t)\|_{\mathcal{H}^m} )\|b^{n+1}(t)\|_{\mathcal{H}^m}^2.
\end{aligned}\end{align}

For the third term on the left-hand side of \eqref{est-b_y0}, we have
\begin{align*}
	\begin{aligned}
		&\left|\int_{\Omega}\partial^\alpha\left( (b^n+B)\partial_x v^{n+1}\right)\cdot\partial^\alpha b^{n+1} dxdy\right|\\
		&= \left|\int_{\Omega}\left\{(b^n+B)\partial_x \partial^{\alpha} v^{n+1}+\left[\partial^\alpha, (b^n+B)\right]\partial_x v^{n+1}\right\}\cdot\partial^\alpha b^{n+1} dxdy\right|\\
	&\leq \epsilon\|\partial_x \partial^{\alpha} v^{n+1}(t)\|_{L^2(\Omega)}^2+C_\epsilon\|(b^n+B)(t)\|_{L^\infty(\Omega)}^2\|\partial^{\alpha} b^{n+1}(t)\|_{L^2(\Omega)}^2\\
	&\quad+\left\|\left[\partial^\alpha, (b^n+B)\right]\partial_x v^{n+1}(t)\right\|_{L^2(\Omega)}\left\|\partial^\alpha b^{n+1}(t)\right\|_{L^2(\Omega)},
	\end{aligned}
\end{align*}
for some sufficiently small $\epsilon>0$.
Similarly to   \eqref{est-commutator2}, one has
\begin{align*}
\displaystyle	\left\|\left[\partial^\alpha, (b^n+B)\right]\partial_x v^{n+1}(t)\right\|_{L^2(\Omega)}\lesssim &(1+\|b^n(t)\|_{\mathcal{H}^m} )\|v^{n+1}(t)\|_{\mathcal{H}^m}.
\end{align*}
Consequently, for $\alpha_2\geq1$ in $\alpha$, it follows, from the above two estimates, that
\begin{align}\label{est-b_y2}
	\begin{aligned}
		&\left|\int_{\Omega}\partial^\alpha\left( (b^n+B)\partial_x v^{n+1}\right)\cdot\partial^\alpha b^{n+1} dxdy\right|\\
	&\leq \epsilon\|\partial_y v^{n+1}(t)\|_{\mathcal{H}^m}^2+C_\epsilon\left(1+\|b^n(t)\|_{\mathcal{H}^m}^2 \right)\|{\bf v}^{n+1}(t)\|_{\mathcal{H}^m}^2.
		\end{aligned}
\end{align}

Moreover, by the similar arguments in \eqref{est-commutator4}, we have
\begin{align*}
	\left\|\partial^{\alpha}\left(B_x v^{n+1}-\phi U_x b^{n+1}\right)\right\|_{L^2(\Omega)}\lesssim \left\|{\bf v}^{n+1}(t)\right\|_{\mathcal{H}^m(\Omega)}.
\end{align*}
It implies
\begin{align}\label{est-b_y3}
\begin{aligned}
&\left|\int_{\Omega}\partial^{\alpha}\left(B_x v^{n+1}-\phi U_x b^{n+1}\right)\cdot\partial^\alpha b^{n+1} dxdy\right|\\
&	\leq \left\|\partial^{\alpha}\left(B_x v^{n+1}-\phi U_x b^{n+1}\right)(t)\right\|_{L^2(\Omega)}\left\|\partial^{\alpha} b^{n+1}(t)\right\|_{L^2(\Omega)}
\lesssim \left\|{\bf v}^{n+1}(t)\right\|_{\mathcal{H}^m(\Omega)}^2.
\end{aligned}\end{align}
Substituting  \eqref{est-b_y1}-\eqref{est-b_y3} into \eqref{est-b_y0}, it follows that
\begin{align*}
	\begin{aligned}
		\frac{d}{dt}\left\|\partial^{\alpha} b^{n+1}(t)\right\|_{L^2(\Omega)}^2\leq &\epsilon\|\partial_y v^{n+1}(t)\|_{\mathcal{H}^m}^2+\|\partial^{\alpha} r_1(t)\|_{L^2(\Omega)}^2+C_\epsilon\left(1+\|{\bf v}^n(t)\|_{\mathcal{H}^m}^2 \right)\|{\bf v}^{n+1}(t)\|_{\mathcal{H}^m}^2.
	\end{aligned}
\end{align*}
Summing the above estimates over $\alpha$ for $|\alpha|\leq m$ and $\alpha_2\geq1$ yields
\begin{align}\label{est-b_y}
	\begin{aligned}
		\frac{d}{dt}\left\|\partial_y b^{n+1}(t)\right\|_{\mathcal{H}^{m-1}}^2\leq &C\epsilon\|\partial_y v^{n+1}(t)\|_{\mathcal{H}^{m}}^2+\| r_1(t)\|_{\mathcal{H}^{m}}^2+C_\epsilon\left(1+\|{\bf v}^n(t)\|_{\mathcal{H}^m}^2 \right)\|{\bf v}^{n+1}(t)\|_{\mathcal{H}^m}^2.
	\end{aligned}
\end{align}

Combining \eqref{est-v_y}, \eqref{est-tan-y} and \eqref{est-b_y} together, and choosing $\epsilon$ suitably small, it holds that
\begin{align}\label{est_y}
\begin{aligned}
&\frac{d}{dt}\left(\left\|{\bf v}^{n+1}(t)\right\|_{\mathcal{B}^{m,0}}^2+ \int_{\Omega}\sum_{|\alpha|\leq m-1}\left[(b^n+B)\partial_{\tau}^{\alpha}\partial_y v^{n+1}\right]^2dxdy+\left\|\partial_y b^{n+1}(t)\right\|_{\mathcal{H}^{m-1}}^2\right)+\|\partial_y v^{n+1}(t)\|_{\mathcal{H}^{m}}^2\\
&\lesssim \| r_1(t)\|_{\mathcal{H}^{m}}^2+\| r_2(t)\|_{\mathcal{H}^{m}}^2+\left(1+\|{\bf v}^n(t)\|_{\mathcal{H}^m}^4 \right)\|{\bf v}^{n+1}(t)\|_{\mathcal{H}^m}^2
\end{aligned}
\end{align}
for $m\geq4$.

Integrating (\ref{est_y}) with respect to $t\in[0,T_{n+1})$  yields that
\begin{align}\label{E9}
\begin{aligned}
	&\left\|{\bf v}^{n+1}(t)\right\|_{\mathcal{B}^{m,0}}^2+ \int_{\Omega}\sum_{|\alpha|\leq m-1}\left[(b^n+B)\partial_{\tau}^{\alpha}\partial_y v^{n+1}\right]^2(t)dxdy+\left\|\partial_y b^{n+1}(t)\right\|_{\mathcal{H}^{m-1}}^2+\int_0^t\|\partial_y v^{n+1}(s)\|_{\mathcal{H}^{m}}^2ds\\
		&\leq \left\|{\bf v}^{n+1}(0)\right\|_{\mathcal{B}^{m,0}}^2+ \int_{\Omega}\sum_{|\alpha|\leq m-1}\left[(b^n+B)\partial_{\tau}^{\alpha}\partial_y v^{n+1}\right]^2(0)dxdy+\left\|\partial_y b^{n+1}(0)\right\|_{\mathcal{H}^{m-1}}^2\\
&\quad+C  \int_0^t\| r_1(s)\|_{\mathcal{H}^{m}}^2+\| r_2(s)\|_{\mathcal{H}^{m}}^2+\left(1+\|{\bf v}^n(s)\|_{\mathcal{H}^m}^4 \right)\|{\bf v}^{n+1}(s)\|_{\mathcal{H}^m}^2ds.
\end{aligned}	\end{align}
By  \eqref{regular-t} and \eqref{initial_n+1}, we have
\begin{align*}
	\left\|{\bf v}^{n+1}(0)\right\|_{\mathcal{B}^{m,0}}^2+ \int_{\Omega}\sum_{|\alpha|\leq m-1}\left[(b^n+B)\partial_{\tau}^{\alpha}\partial_y v^{n+1}\right]^2(0)dxdy+\left\|\partial_y b^{n+1}(0)\right\|_{\mathcal{H}^{m-1}}^2\lesssim M_0.
\end{align*}
Then, by applying the above estimate, the lower boundedness of $b^n+B$ in \eqref{iteration-n} and \eqref{est-r} into \eqref{E9},  there exists a constant $C=C(T, M,M_0, \delta_0)$ such that
\begin{align*}
\begin{aligned}
	&\left\|{\bf v}^{n+1}(t)\right\|_{\mathcal{B}^{m,0}}^2+\left\|\partial_y v^{n+1}(t)\right\|_{\mathcal{B}^{m-1,0}}^2+\left\|\partial_y b^{n+1}(t)\right\|_{\mathcal{H}^{m-1}}^2+\int_0^t\|\partial_y v^{n+1}(s)\|_{\mathcal{H}^{m}}^2ds\\
		&\leq C+ C \int_0^t\left(1+\|{\bf v}^n(s)\|_{\mathcal{H}^m}^4 \right)\|{\bf v}^{n+1}(s)\|_{\mathcal{H}^m}^2ds,
\end{aligned}	\end{align*}
or
\begin{align}\label{Est_y}
\begin{aligned}
	&\left\| v^{n+1}(t)\right\|_{\mathcal{B}^{m,0}}^2+\left\|\partial_y v^{n+1}(t)\right\|_{\mathcal{B}^{m-1,0}}^2+\left\| b^{n+1}(t)\right\|_{\mathcal{H}^{m}}^2+\int_0^t\|\partial_y v^{n+1}(s)\|_{\mathcal{H}^{m}}^2ds\\
		&\leq C+ C \int_0^t\left(1+\|{\bf v}^n(s)\|_{\mathcal{H}^m}^4 \right)\|{\bf v}^{n+1}(s)\|_{\mathcal{H}^m}^2ds.
\end{aligned}	\end{align}
Furthermore, from \eqref{est-v_yy2} and \eqref{Est_y},  we obtain
\begin{align}\label{Est_final}
\begin{aligned}
	&\left\|{\bf v}^{n+1}(t)\right\|_{\mathcal{H}^{m}}^2+\int_0^t\|\partial_y v^{n+1}(s)\|_{\mathcal{H}^{m}}^2ds\leq C+ C \int_0^t\left(1+\|{\bf v}^n(s)\|_{\mathcal{H}^m}^4 \right)\|{\bf v}^{n+1}(s)\|_{\mathcal{H}^m}^2ds.
\end{aligned}	\end{align}

Now, 
we are ready to establish the uniform in $n$ boundedness of $\|{\bf v}^{n+1}(t)\|_{\mathcal{H}^m}$ in a small time interval. Precisely,
by applying the Gronwall inequality to \eqref{Est_final}, it shows that  there is a constant $C>0$ depending only on $T, \delta_0, M$ and $M_0$  such that
\begin{align}\label{est-final}
	\|{\bf v}^{n+1}(t)\|^2_{\mathcal{H}^{m}(\Omega) }+\int_0^t\|\partial_y v^{n+1}(s)\|_{\mathcal{H}^{m}}^2ds\leq C\exp\left\{C\int_0^t\|{\bf v}^n(s)\|_{\mathcal{H}^m}^4ds\right\}.
\end{align}
Thus, choosing suitably large $M_\ast$, there exists a sufficient small $\tilde T>0$ which depends only on $C$ and $M_\ast$  such that
\begin{align}\label{est-fin}
{\bf v}^n\in E_{\tilde T}(M_\ast) ~\Longrightarrow~  \|{\bf v}^{n+1}\|_{\mathcal{H}^m(\Omega_{\tilde T})}\leq M_\ast.
\end{align}
More precisely, denote by
\begin{align}\label{def-C0}
	C_0~=~\max\left\{C,~\sup_{t\in[0,T_0]}\|{\bf v}^0(t)\|_{\mathcal{H}^m } \right\}
\end{align}
with $C$ given in \eqref{est-final}, we have the following proposition.
\begin{prop}
	Let $\displaystyle T^\ast=\min\left\{T_0, \hat T\right\}$ with $\hat T$ given by
	\begin{align}\label{def-hatT}
		\hat T\left(M+\frac{\sqrt{C_0}}{\sqrt[4]{1-2C_0^3\hat T}}\right)=\frac{\delta_0}{2},
	\end{align}
	where $M$ and $C_0$ are given in \eqref{OUTE} and \eqref{def-C0} respectively. The approximate solution sequence $\{{\bf v}^n\}_{n=0}^\infty$ constructed in the subsection \ref{SUB} satisfies that
	\begin{align}\label{est-prop1}
		\left\|{\bf v}^n(t) \right\|_{\mathcal{H}^m}^2~\leq~\frac{C_0}{\sqrt{1-2C_0^3t}},\qquad\forall~t\in[0,T^\ast],
	\end{align}
	and
	\begin{align}\label{est-prop2}
		b^n(t,x,y)+B(t,x)\geq\frac{\delta_0}{2}, \qquad \forall~(t,x,y)\in\Omega_{T^\ast}
	\end{align}
for any $n\geq0$.
\end{prop}
\begin{proof}
	We will prove this proposition by induction of $n$. First of all, by using the definition of $C_0$ and $T^\ast$, it is easy to know that \eqref{est-prop1} and \eqref{est-prop2} hold for $n=0$.
	
Assume that the estimates \eqref{est-prop1} and \eqref{est-prop2} are valid for some $n\geq0$, then
as what are shown  in the above three subsections \ref{SUB1}-\ref{SUB3}, we  obtain \eqref{est-final} for ${\bf v}^{n+1}$, in which the constant $C$ is replaced by $C_0$. Therefore, by the induction hypothesis,  it follows that,
\begin{align}\label{est-uni-n+1}
	\|{\bf v}^{n+1}(t)\|^2_{\mathcal{H}^{m}(\Omega) }\leq C_0\exp\left\{C_0\int_0^t\frac{C_0^2}{1-2C_0^3s}ds\right\}=\frac{C_0}{\sqrt{1-2C_0^3t}}
\end{align}
for $t\in[0,T^\ast]$. 

In addition,  since
\begin{align*}
	b^{n+1}(0,x,y)+B(0,x)=b_0(x,y)\geq\delta_0
\end{align*}
due to \eqref{IID}, it follows that
\begin{align*}\begin{aligned}
	b^{n+1}(t,x,y)+B(t,x)=&b^{n+1}(0,x,y)+B(0,x)+\int_0^t\partial_t b^{n+1}(s,x,y)+\partial_t B(s,x)ds\\
	\geq &\delta_0-\int_0^t\|\partial_t b^{n+1}(s,\cdot)\|_{L^\infty(\Omega)}+\|\partial_t B(s,\cdot)\|_{L^\infty(\T_x)}ds\\
	\geq &\delta_0-\int_0^t\|b^{n+1}(s)\|_{\mathcal{H}^3}+\|\partial_t B(s,\cdot)\|_{H^1(\T_x)}ds
\end{aligned}\end{align*}
for any $t\in[0,T^\ast]$, where the Sobolev embedding inequalities are used. Then, together with \eqref{OUTE} and \eqref{est-uni-n+1}, this implies
\begin{align*}
	b^{n+1}(t,x,y)+B(t,x)\geq \delta_0-t\left(\frac{\sqrt{C_0}}{\sqrt[4]{1-2C_0^3t}}+M\right), \quad \forall (t,x,y)\in \Omega_{T^\ast}.
\end{align*}
By $T^\ast\leq\hat T$ and the definition of $\hat T$ in \eqref{def-hatT},   \eqref{est-prop2} holds for ${\bf v}^{n+1}$. Consequently, the proof of this proposition is completed.
\end{proof}
From the above proposition, we  obtain for   $T^\ast$ and
\begin{align}\label{def-M*}
	M_\ast~=~\frac{C_0}{\sqrt{1-2C_0^3T^\ast}},
\end{align}
 the mapping $\Pi$ defined by \eqref{def-map} and \eqref{iteration-n+1} is a mapping from $E_{T^\ast}(M_\ast)$ to $E_{T^\ast}(M_\ast)$ itself. In particular, it implies that the approximate solutions $\{{\bf v}^n\}_{n=0}^\infty$ constructed by \eqref{iteration-0} and \eqref{iteration-n+1} are uniform bounded in $\mathcal{H}^m(\Omega_{T^\ast})$ with $T^\ast>0$ independent of $n$,
\begin{align}\label{uni-bound}
	\|{\bf v}^n\|_{\mathcal{H}^m(\Omega_{T^\ast})}~\leq ~M_\ast,\qquad \forall\ n.
\end{align}

\subsection{Contraction in $L^2$ norm}
\label{UNIQ}
In this subsection, we will show the linear mapping $\Pi$ defined in (\ref{def-map}) is a contraction mapping in $L^2$ sense, at least in a  small time interval.

Set
\begin{align*}
\bar{\bf v}^{n+1}=( \bar{b}^{n+1}, \bar{v}^{n+1})=\left(b^{n+1}-b^n, v^{n+1}-v^n\right),\qquad n\geq1.
\end{align*}
Then by  \eqref{iteration-n+1}-\eqref{BBB},  $\bar{\bf v}^{n+1}(t,x,y)$ satisfies
\begin{align}\label{DE}
\left\{
\begin{array}{ll}
\partial_t\bar{b}^{n+1}+(v^n+U\phi)\partial_x \bar{b}^{n+1}-(b^n+B)\partial_x\bar{v}^{n+1}+B_x\bar{v}^{n+1}-\phi U_x\bar{b}^{n+1}=R_1^n,\\
\partial_t\bar{v}^{n+1}-(b^n+B)\partial_x\bar{b}^{n+1}+(v^n+U\phi)\partial_x \bar{v}^{n+1}-(b^n+B)\partial_y\left((b^n+B)\partial_y\bar{v}^{n+1}+\phi_y U\bar{b}^{n+1} \right)\\
\qquad-(\phi_{yy}UB+B_x)\bar{b}^{n+1}+\phi U_x\bar{v}^{n+1}=R_2^n,\\
\end{array}
\right.
\end{align}
with
\begin{align*}\begin{cases}
R_1^n=-\bar{v}^n\partial_x b^n+\bar{b}^n\partial_x v^n,\\
R_2^n=\bar{b}^n\partial_x b^n-\bar{v}^n\partial_x v^n+\bar{b}^n\left[\partial_y\left((b^n+B)\partial_y v^n+ \phi_y Ub^n \right)+\left(b^{n-1}+B\right)\partial_y^2 v^n \right]+(b^{n-1}+B)\partial_y v^n\partial_y\bar{b}^n.
\end{cases}\end{align*}
As in subsection \ref{SUB}, we can obtain the following estimate similar to \eqref{est-L2},
\begin{align}\label{DL}
\begin{aligned}
&\frac{d}{dt}\left\|\bar{\bf v}^{n+1}(t)\right\|_{L^2(\Omega)}^2+\left\|\partial_y \bar v^{n+1}(t)\right\|_{L^2(\Omega)}^2\\
&\leq C\big(1+\|{\bf v}^n(t)\|^2_{\mathcal{H}^3}\big)\|\bar{\bf v}^{n+1}(t)\|^2_{L^2(\Omega)}+C\int_{\Omega}(R_1^n\cdot\bar{b}^{n+1}+R_2^n\cdot\bar{v}^{n+1})dxdy, 
\qquad t\in[0,T^\ast].
\end{aligned} \end{align}
Denote
\begin{align*}
	R_2^n=(b^{n-1}+B)\partial_y v^n\partial_y\bar{b}^n+\tilde R_2^n.
\end{align*}
Then by  the boundary condition $\bar v^{n+1}|_{y=0}=0$, we get
\begin{align*}
	\int_{\Omega}R_2^n\cdot\bar v^{n+1}dxdy=&-\int_{\Omega}\bar b^n\left\{(b^{n-1}+B)\partial_y v^n\partial_y\bar v^{n+1}+\partial_y\left((b^{n-1}+B)\partial_y v^n\right\}\bar v^{n+1}  \right)dxdy\\
	&+\int_{\Omega}\tilde R_2^n\cdot \bar v^{n+1}dxdy\\
	\leq &\frac{1}{2C}\|\partial_y\bar v^{n+1}(t)\|_{L^2(\Omega)}^2+\frac{C}{2}\|(b^{n-1}+B)\partial_y v^n(t)\|_{L^\infty(\Omega )}^2\|\bar b^n(t)\|_{L^2(\Omega)}^2\\
	&+\int_{\Omega}\left(\tilde R_2^n-\bar b^n\partial_y\left((b^{n-1}+B)\partial_y v^n\right)\right)\cdot \bar v^{n+1}dxdy.
\end{align*}
Furthermore, based on the definition of $R_1^n$ and $R_2^n$ and the Sobolev embedding inequality, we have
\begin{align*}
	\|R_1^n(t)\|_{L^2(\Omega)}\leq \|{\bf v}^n(t)\|_{W^{1,\infty}(\Omega)}\|\bar{\bf v}^n(t)\|_{L^2(\Omega )}\leq \|{\bf v}^n(t)\|_{\mathcal{H}^{3}}\|\bar{\bf v}^n(t)\|_{L^2(\Omega )},
\end{align*}
and
\begin{align*}
	&\left\|\left(\tilde R_2^n-\bar b^n\partial_y\left((b^{n-1}+B)\partial_y v^n\right)\right)(t)\right\|_{L^2(\Omega)}\\
	&\lesssim \left(1+\|b^n(t)\|_{W^{1,\infty}(\Omega)}+\|b^{n-1}(t)\|_{W^{1,\infty}(\Omega)} \right)\|{\bf v}^n(t)\|_{W^{2,\infty}(\Omega)}\|\bar{\bf v}^n(t)\|_{L^2(\Omega )}\\
	&\lesssim \left(1+\|{\bf v}^{n-1}(t)\|_{\mathcal{H}^{3}}^2+\|{\bf v}^n(t)\|_{\mathcal{H}^{4}}^2\right)\|\bar{\bf v}^n(t)\|_{L^2(\Omega )}.
\end{align*}
Consequently,
\begin{align*}
	\begin{aligned}
		&C\int_{\Omega}R_1^n\cdot\bar{b}^{n+1}+R_2^n\cdot\bar{v}^{n+1}dxdy\\
		&\leq \frac12\|\partial_y\bar v^{n+1}(t)\|_{L^2(\Omega)}^2+C_1\|\bar{\bf v}^{n+1}(t)\|_{L^2(\Omega)}^2+C_1\left(1+\|{\bf v}^{n-1}(t)\|_{\mathcal{H}^{3}}^4+\|{\bf v}^n(t)\|_{\mathcal{H}^{4}}^4\right)\|\bar{\bf v}^n(t)\|_{L^2(\Omega )}^2.
	\end{aligned}
\end{align*}
Substituting the above estimate into \eqref{DL} and using the uniform estimate \eqref{uni-bound} for $\|{\bf v}^n(t)\|_{\mathcal{H}^m}$ with $m\geq4$, it is found that there exists a constant $C_2>0$, depending only on $T, \delta_0, M, M_0$ and $M_\ast$, such that
\begin{align*}
\begin{aligned}
&\frac{d}{dt}\left\|\bar{\bf v}^{n+1}(t)\right\|_{L^2(\Omega)}^2+\left\|\partial_y \bar v^{n+1}(t)\right\|_{L^2(\Omega)}^2\leq C_2\big(\|\bar{\bf v}^n(t)\|^2_{L^2(\Omega)}+\|\bar{\bf v}^{n+1}(t)\|^2_{L^2(\Omega)}\big),
\qquad t\in[0,T^\ast].
\end{aligned} \end{align*}
By applying the Gronwall inequality on the above estimate, it follows that
\begin{align*}
	\sup_{0\leq s\leq t}\left\|\bar{\bf v}^{n+1}(s)\right\|_{L^2(\Omega)}^2\leq C_2e^{C_2t}\cdot\int_0^t\|\bar{\bf v}^n(s)\|_{L^2(\Omega)}^2ds\leq C_2 te^{C_2t}\cdot\sup_{0\leq s\leq t}\left\|\bar{\bf v}^{n}(s)\right\|_{L^2(\Omega)}^2
\end{align*}
for any $t\in[0,T^\ast]$. Let
\begin{align}\label{def-T**}
T_{\ast}~=~\min\left\{T^\ast, ~\frac{1}{2C_2e^{C_2T^\ast}}\right\},	
\end{align}
we obtain
\begin{align}\label{contract}
	\sup_{0\leq s\leq t}\left\|\bar{\bf v}^{n+1}(s)\right\|_{L^2(\Omega)}^2\leq \frac{1}{2}\sup_{0\leq s\leq t}\left\|\bar{\bf v}^{n}(s)\right\|_{L^2(\Omega)}^2,\quad \forall~t\in[0,T_{\ast}],
\end{align}
which indeed shows that the mapping \eqref{def-map} is contractive in $L^2$.

\subsection{Proof of Theorem \ref{T3.1}}
Based on the uniform estimates in (\ref{uni-bound}) and contraction property of (\ref{contract}), it follows that the approximate solution sequence $\{{\bf v}^{n}\}_{n=0}^\infty$ is a Cauchy sequences in $\mathcal{H}^{k}(\Omega_{T_{\ast}})$ with $k<m$. Then  $\{{\bf v}^{n}\}_{n=0}^\infty$  converges in $\mathcal{H}^{k}(\Omega_{T_{\ast}})$ strongly. That is, there exists ${\bf v}=(b, v)$ such that 
\begin{align*}
\lim_{n\rightarrow+\infty}{\bf v}^n={\bf v},\quad \hbox{in}\quad \mathcal{H}^k(\Omega_{T_{\ast}}),
\end{align*}
with
\begin{align}\label{UEGU}
\|{\bf v}\|_{\mathcal{H}^k(\Omega_{T_{\ast}})}\leq M_\ast.
\end{align}
Let $n\rightarrow+\infty$ in (\ref{iteration-n+1}) and set
\begin{align*}
b_1=b+B(t,x),\qquad u=v+U(t,x)\phi(y).
\end{align*}
Then $(b_1, u)$ solves (\ref{VBE}) uniquely and this completes the proof of the Theorem \ref{T3.1}.

\subsection{Proof of Theorem \ref{MTH}}
With the solution $(b_1, u)$ to (\ref{VBE}) in hand,  we define a function $\psi=\psi(t,x,y)$ in the following form.
\begin{equation}\label{invers}
y=\int^{\psi(t,x,y)}_0\frac{d\eta}{b_1(t,x,\eta)}.
\end{equation}
It is straightforward  to know that $\psi(t,x,y)$ is well-defined in $\Omega_{T_\ast}$, because  $b_1\geq\frac{\delta_0}{2}$ given in \eqref{positi}.

Set
\begin{align}
(\hat{u}, \hat{b}_1)(t,x,y)=(u, b_1)(t,x,\psi(t,x,y)),
\end{align}
and
\begin{equation}\label{(u2,h2)-trans}
\hat{v}(t,x,y)=-\frac{(\partial_t\psi+\hat{u}\partial_x\psi)}{\hat{b}_1}(t,x,y),\quad \hat{b}_2(t,x,y)=-\partial_x\psi(t,x,y).
\end{equation}
Then, $(\hat{u}, \hat{v}, \hat{b}_1, \hat{b}_2)(t,x,y)$  is the desired solution of Theorem \ref{MTH} and
the  proof of Theorem \ref{MTH} is thus completed.

\bigskip

\section{Linear ill-posedness}
In this section, we will prove Theorem \ref{thm_lin} to show the linear instability of MHD boundary layer when the tangential magnetic field is degenerate at one point.
Let us recall  the important work \cite{G-D} about the linear ill-posedness of Prandtl equation without the monotonicity condition, where the key ingredient is to construct a strong unstable approximate solution to the linearized Prandtl equation. For later use, we present first the construction of approximate solution to the velocity field given in \cite{G-D}.  Without loss of generality, we assume that $U_s''(a)<0$, so that
the differential equation
\begin{equation}\label{def_alpha}\begin{cases}
\partial_t\partial_y u_s\big(t,a(t)\big)+\partial_y^2u_s\big(t,a(t)\big)a'(t)=0,\\
a(0)~=~a,
\end{cases}\end{equation}
defines a non-degenerate critical point curve $y=a(t)$ of $u_s(t,\cdot)$ satisfying $\partial_y u_s(t, a(t))=0$ and $\partial_y^2u_s\big(t,a(t)\big)<0$ for all $t\in[0,t_0)$ with $t_0$ being small enough.
As proved in \cite{G-D}, there exists a pair $\big(\tau, W(z)\big)$ such that the complex number $\tau$ satisfies $\Im \tau<0$ and the smooth function $W(z)$ satisfies
\begin{equation}\label{SC}\begin{cases}
\big(\tau-z^2\big)^2\frac{d}{dz}W+i\frac{d^3}{dz^3}\big((\tau-z^2)W\big)=0,\\
\lim\limits_{z\rightarrow-\infty}W(z)~=~0,\quad\lim\limits_{z\rightarrow+\infty}W(z)~=~1.
\end{cases}\end{equation}
Set
\begin{equation}\begin{split}\label{def_phi}
V(z)~&:=~\big(\tau-z^2\big) W( z)-\mathbf1_{\R^+}\big(\tau-z^2\big).
\end{split}\end{equation}
The approximate solution of the linearized Prandtl equation is defined as 
\begin{equation}\label{ap_lin}
(u_\ep,v_\ep)(t,x,y)=e^{i\ep^{-1}x}\big(U_\ep,V_\ep\big)(t,y)
\end{equation}
with
\begin{equation}\label{expre_uv}
U_\ep(t,y)~=~ie^{i\ep^{-1}\int_0^t \omega(\epsilon, s)ds}\partial_y W_\epsilon(t,y),\quad
V_\ep(t,y)~=~\epsilon^{-1}e^{i\ep^{-1}\int_0^t \omega(\epsilon,s)ds} W_\epsilon(t,y),
\end{equation}
where
\begin{equation}\label{app_lam}
\omega(\epsilon, t)~:=~-u_s\big(t,a(t)\big)+\sqrt{\epsilon}\Big|\frac{\partial_y^2u_s\big(t,a(t)\big)}{2}\Big|^{\frac{1}{2}}~\tau,
\end{equation}
and
\begin{align}\label{def-W}
\begin{aligned}
	W_\epsilon(t,y)=~&H\big(y-a(t)\big)\Big[u_s(t,y)-u_s\big(t,a(t)\big)+\sqrt{\epsilon}\Big|\frac{\partial_y^2u_s\big(t,a(t)\big)}{2}\Big|^{\frac{1}{2}} \tau\Big]\\
	&+\sqrt{\epsilon}\varphi\big(y-a(t)\big)\Big|\frac{\partial_y^2u_s\big(t,a(t)\big)}{2}\Big|^{\frac{1}{2}} ~V\Big(\Big|\frac{\partial_y^2u_s\big(t,a(t)\big)}{2}\Big|^{\frac{1}{4}}\cdot\frac{y-a(t)}{\ep^{\frac{1}{4}}}\Big)\\
	= 
	&v_\epsilon^{reg}(t,y)+v_\epsilon^{sl}(t,y).
\end{aligned}
	\end{align}
Here $H(\cdot)$ is the Heaviside function and $\varphi(\cdot)$ is a smooth truncation function near 0.
	 To make the function $(u_\ep,v_\ep)(t,x,y)$ in \eqref{ap_lin} to be $2\pi-$periodic in $x$, we take $\ep=\frac{1}{n}$ with  $n\in \mathbb{N}$. It is straightforward to check that, 
\[(u_\ep,v_\ep)|_{y=0}=0,\qquad\lim\limits_{y\to+\infty} u_\ep=0,\]
 and the divergence-free condition also holds for $(u_\epsilon, v_\epsilon)$. And $u_\ep(t,x,y)=e^{i\ep^{-1}x}U_\ep(t,y)$ is analytic in the tangential variable $x$ and $W^{2,\infty}$ in $y$. Moreover, $(U_\epsilon,V_\epsilon)$ has the growing mode like $e^{-\Im\tau\cdot\frac{t}{\sqrt{\epsilon}} }$, i.e., there are positive constants $C_0$ and $\sigma_0$, independent of $\ep$, such that
\begin{equation}\label{bound_app}
C_0^{-1}e^{\frac{\sigma_0t}{\sqrt{\ep}}}~\leq~
\|U_\ep(t,\cdot)\|_{W_\alpha^{2,\infty}}
~\leq~C_0e^{\frac{\sigma_0t}{\sqrt{\ep}}},\quad t\in[0,t_0),
\end{equation}
which is the key of instability mechanism.

Now we construct the approximate solution of problem \eqref{linear_pr} by choosing $(u_\epsilon, v_\epsilon, b_{1,\epsilon}, b_{2,\epsilon})(t,x,y)$ in \eqref{linear_pr} to be the following form,
\begin{align}\label{def-app}
	\begin{aligned}
		(u_\epsilon, v_\epsilon, b_{1,\epsilon}, b_{2,\epsilon})(t,x,y)=e^{i\epsilon^{-1}x}\left(U_\epsilon, V_\epsilon, B_{1,\epsilon}, B_{2,\epsilon}\right)(t,y).
	\end{aligned}
\end{align}
 On one hand, $(U_\epsilon, V_\epsilon)(t, y)$ is taken to be the same form as given by \eqref{expre_uv}-\eqref{def-W} to preserve the instability mechanism; on the other hand, $(B_{1,\epsilon}, B_{2,\epsilon})(t,x,y)$ are taken as the following,
\begin{align}\label{def-HG}
B_{1,\epsilon}(t,y)=ie^{i\epsilon^{-1}\int_0^t\omega(\epsilon,s)ds}\partial_y\Phi_\epsilon(t,y),\qquad B_{2,\epsilon}(t,y)=\epsilon^{-1} e^{i\epsilon^{-1}\int_0^t\omega(\epsilon,s)ds}\Phi_\epsilon(t,y)
\end{align}
with
\begin{align}\label{def-Phi}
\displaystyle\Phi_\epsilon(t,y)=\frac{b_s(y)W_\epsilon(t,y)}{\omega(\epsilon,t)+u_s(t,y)}=b_s(y)\cdot\frac{v_\epsilon^{reg}(t,y)+v_\epsilon^{sl}(t,y) }{\omega(\epsilon,t)+u_s(t,y)},
\end{align}
where $\omega(\epsilon,t)$ and $W_\epsilon(t,y)$ are given in \eqref{app_lam} and \eqref{def-W} respectively.  It is straightforward to show that 
\[(u_\epsilon,v_\epsilon, b_{2,\epsilon})|_{y=0}=0,\qquad\lim\limits_{y\to+\infty} (u_\epsilon, b_{1,\epsilon})=\mathbf0,\]
 and the divergence-free conditions are satisfied. Also, $(u_\epsilon, b_{1,\epsilon})(t,x,y)=e^{i\ep^{-1}x}\big(U_\epsilon, B_{1,\epsilon})(t,y)$ is analytic in the tangential variable $x$ and is in $W^{2,\infty}$ in $y$. Furthermore,  $\big(U_\epsilon, V_\epsilon, B_{1,\epsilon}, B_{2,\epsilon}\big)(t,y)$ still preserves the growing mode in $t$, i.e., there are positive constants $C_0$ and $\sigma_0$, independent of $\ep$, such that
\begin{equation}\label{bound_app}
C_0^{-1}e^{\frac{\sigma_0t}{\sqrt{\ep}}}~\leq~
\big\|\big(U_\epsilon, B_{1,\epsilon})(t,\cdot)\big\|_{W_\alpha^{2,\infty}}
~\leq~C_0e^{\frac{\sigma_0t}{\sqrt{\ep}}},\quad t\in[0,t_0).
\end{equation}

Substituting the approximate solution  \eqref{def-app} into the problem \eqref{linear_pr}, it follows that
\begin{equation}\label{lin_ap}\begin{cases}
\partial_t u_\ep+u_s\partial_x u_\ep+v_\ep \partial_y u_s-\partial_y^2u_\ep-b_s\partial_x b_{1,\epsilon}-b_{2,\epsilon}  b'_s =r_\ep^1,\quad&\\
\partial_t b_{1,\epsilon}+u_s\partial_x b_{1,\epsilon}+v_\ep b'_s-b_s\partial_x u_\epsilon-b_{2,\epsilon} \partial_y u_s =r_\ep^2,\quad&\\
\partial_x u_\epsilon+\partial_y v_\epsilon=0,\quad\partial_x b_{1,\epsilon}+\partial_y b_{2,\epsilon}=0,&{\rm} \\ 
(u_\epsilon,v_\epsilon, b_{2,\epsilon})|_{y=0}=0,
\end{cases}\end{equation}
in $\Omega$.
The remainder term $(r_\epsilon^1,r_\epsilon^2)$ can be represented by
$$(r_\epsilon^1,r_\epsilon^2)(t,x,y)=e^{i\ep^{-1}x}\big(R_\epsilon^1, R_\epsilon^2)(t,y),$$
where
\begin{equation}\label{tr}\begin{split}
R_\epsilon^1(t,y)=e^{i\ep^{-1}\int_0^t w(\epsilon,s)ds}\Big\{&-\ep^{-1}\Big[u_s(t,y)
-u_s(t,a(t))-\partial_y^2u_s(t,a(t))\frac{(y-a(t))^2}{2}\Big]\partial_yv_\ep^{sl}(t,y)\\
&+\ep^{-1}\Big[\partial_y u_s(t,y)-\partial_y^2u_{s}\big(t,a(t)\big)\cdot\big(y-a(t)\big)\Big]v_\ep^{sl}(t,y)\\
&+\epsilon^{-1}b_s^2(y)\left[\frac{\partial_y v_\epsilon^{sl}(t,y)}{\omega(\epsilon,t)+u_s(t,y)}-\frac{\partial_y u_s(t,y)}{\big(\omega(\epsilon,t)+u_s(t,y)\big)^2}v_\epsilon^{sl}(t,y)\right] \\
&+i\partial_t\partial_y v_\ep^{sl}(t,y)+O(\ep^\infty e^{\frac{\sigma_0t}{\sqrt{\epsilon}}} )\Big\},
\end{split}\end{equation}
and
\begin{equation}\label{br}\begin{split}
R_\epsilon^2(t,y)&=ie^{i\ep^{-1}\int_0^t w(\epsilon,s)ds}\partial_{ty}^2\Phi(t,y)\\
&=ie^{i\ep^{-1}\int_0^t w(\epsilon,s)ds}\partial_{t}\left\{b_s(y)\left[\frac{\partial_y v_\epsilon^{sl}(t,y)}{\omega(\epsilon,t)+u_s(t,y)}-\frac{\partial_y u_s(t,y)}{\big(\omega(\epsilon,t)+u_s(t,y)\big)^2}v_\epsilon^{sl}(t,y)\right]\right.\\
&\hspace{3.5cm}\left.+b'_s(y)\cdot\frac{v_\epsilon^{reg}(t,y)+v_\epsilon^{sl}(t,y) }{\omega(\epsilon,t)+u_s(t,y)}\right\}.\\
\end{split}\end{equation}
Here, the term $O(\ep^\infty)$ in \eqref{tr} denotes the part of remainder
with exponential decay in $y$   from the fact that $V(z)$ decays exponentially and the derivatives of $\varphi(\cdot-a(t))$ vanish
outside a neighborhood of  $a(t)$.

Let $\displaystyle z=\frac{y-a(t)}{\epsilon^{1/4}}.$ Note that in the vicinity of $\{y=a(t)\},$
\begin{align*}
	\omega(\epsilon, t)+u_s(t,y)=u_s(t, y)-u_s\big(t,a(t)\big)+\sqrt{\epsilon}\Big|\frac{\partial_y^2u_s\big(t,a(t)\big)}{2}\Big|^{\frac{1}{2}}~\tau=O(\sqrt{\epsilon})(1+z^2),
\end{align*}
and
\begin{align*}
	\begin{aligned}
		&b_s(y)=O(1)(y-a)^2=O(1)[y-a(t)+a(t)-a]^2=O(1)(t+\epsilon^{1/4}z )^2,\\
		&b'_s(y)=O(1)(y-a)=O(1)[y-a(t)+a(t)-a]=O(1)(t+\epsilon^{1/4}z ).
	\end{aligned}
\end{align*}
Therefore, Based on the exponential decay properties of $v_\ep^{sl}(t,y)$ and the formulations \eqref{tr} and \eqref{br} of $R_\epsilon^1(t,y)$ and $R_\epsilon^2(t,y)$, we have
\begin{equation}\label{bd_tr}
\big\|\big( R_\epsilon^1, R_\epsilon^2\big) (t,\cdot)\big\|_{W_\alpha^{0,\infty}}\leq C_1 e^{\frac{\sigma_0t}{\sqrt{\ep}}}(\epsilon^{-1/4}+\epsilon^{-5/4}t^4 ),
\quad\forall\alpha\geq0
\end{equation}
with the same constant $\sigma_0>0$ given in \eqref{bound_app}.


\begin{proof}[\bf Proof of  Theorem \ref{thm_lin}]
With the approximate solutions constructed   above, we can
apply the approach  in \cite{G-D} to prove  Theorem \ref{thm_lin}.
The proof relies on the verification of \eqref{est_in} for the tangential differential operator through contradiction.

Suppose that \eqref{est_in} does not hold, that is, for all $\sigma>0$, there exists $\delta>0,~\alpha_0,m\geq0$ and $\mu\in[0,\frac{1}{4})$ such that
\begin{equation}\label{pr_x}
\sup\limits_{0\leq s< t\leq \delta}\|e^{-\sigma(t-s)\sqrt{|\partial_x|}}\mathcal{T}(t,x)\|_{\cl(\h_{\alpha}^m,\h_\alpha^{m-\mu})}<+\infty.
\end{equation}
Introduce the operator
\begin{equation*}
\mathcal{T}_\ep(t,s):~W_{\alpha}^{0,\infty}(\R^+)\mapsto W_{\alpha}^{0,\infty}(\R^+)
\end{equation*}
where
\begin{equation}\label{def_t}
 \mathcal{T}_\ep(t,s)(U_0, B_0)~:=~e^{-i\ep^{-1}x}\mathcal{T}(t,s)\Big(e^{i\ep^{-1}x}(U_0,B_0)\Big)
\end{equation}
with $\mathcal{T}(t,s)$ being defined in \eqref{def_T}.
From \eqref{pr_x}, we have
\begin{equation}\label{est_tep}
\|\mathcal{T}_\ep(t,s)\|_{\cl(W_{\alpha}^{0,\infty},W_\alpha^{0,\infty})}~\leq~C_2\ep^{-\mu}e^{\frac{\sigma(t-s)}{\sqrt{\ep}}},\quad \forall ~0\leq s< t\leq\delta
\end{equation}
for some constant $C_2>0$ independent of $\ep$.

Next, we introduce the operator
\begin{align*}
	L(u(t,x,y), b_1(t,x,y))=\Big(& u_s\partial_x u+v\partial_y u_s -\partial_y^2 u-b_s\partial_x h-b_2b'_s,\\
	&u_s\partial_x b_1+vb'_s -b_s\partial_x u-b_2\partial_y u_s\Big)
\end{align*}
with
\begin{align*}
v(t,x,y)=	-\int_0^y\partial_x u(t,x,\tilde y)d\tilde y,\quad b_2(t,x,y)=-\int_0^y\partial_x b_1(t,x,\tilde y)d\tilde y.
\end{align*}
Denote
\[L_\ep~:=~e^{-i\ep^{-1}x}~L~e^{i\ep^{-1}x},\]
and let $(U,B_1)(t,y)$ be a solution to the problem
\[
\partial_t (U, B_1)+L_\ep (U,B_1)~=~\mathbf0,\qquad
(U,B_1)|_{t=0}~=~(U_\epsilon, B_{1,\epsilon})(0,y),
\]
where $U_\epsilon$ and  $B_{1,\epsilon}(t,y)$ are given in \eqref{expre_uv} and \eqref{def-HG} respectively.
Then, we have
\[(U,B_1)(t,y)~=~\mathcal{T}_\ep(t,0)(U_\epsilon, B_{1,\epsilon})(0,y),\]
and by using \eqref{bound_app} and \eqref{est_tep}, it follows that
\begin{equation}\label{up_bound}
\|(U,B_1)(t,\cdot)\|_{W_\alpha^{0,\infty}}\leq C_2\ep^{-\mu}e^{\frac{\sigma t}{\sqrt{\ep}}}\|(U_\epsilon, B_{1,\epsilon})(0,\cdot)\|_{W_{\alpha}^{0,\infty}}
\leq C_3\ep^{-\mu}e^{\frac{\sigma t}{\sqrt{\ep}}},\quad \forall t\in(0,\delta]
\end{equation}
 for some constant $C_3>0$ independent of $\ep$.

On the other hand, the difference $(\tilde U, \tilde B_1):=(U-U_\ep, B_1-B_{1,\epsilon})$ can be obtained by the Duhamel principle:
\begin{equation}\label{eq_dif}
(\tilde U, \tilde B_1)(t,\cdot)~=~\int_0^t\mathcal{T}_\ep(t,s) (R_\epsilon^1, R_\epsilon^2)(s,\cdot)ds,\quad \forall~t\leq\delta.
\end{equation}
Combining \eqref{bd_tr}, \eqref{est_tep} and \eqref{eq_dif},
and choosing $\sigma<\sigma_0$, we have
\begin{align}\label{est_dif}
\begin{aligned}
\|(\tilde U, \tilde B_1)(t,\cdot)\|_{W^{0,\infty}_0}&\leq C_1C_2\ep^{-\mu}
\int_0^te^{\frac{\sigma(t-s)}{\sqrt{\ep}}}e^{\frac{\sigma_0s}{\sqrt{\ep}}}(\epsilon^{-1/4}+\epsilon^{-5/4}s^4 )ds
\\
&\lesssim \ep^{\frac{1}{4}-\mu}e^{\frac{\sigma_0t}{\sqrt{\ep}}}+\epsilon^{-\mu+\frac{5}{4}}e^{\frac{\sigma_0t}{\sqrt{\ep}}}\big(1+\frac{t^4}{\epsilon^2})\\
&\leq C_4 \epsilon^{\frac{1}{4}-\mu}e^{\frac{\sigma_0t}{\sqrt{\ep}}}\big(1+\frac{t^4}{\epsilon}\big),
\end{aligned}\end{align}
where the constant $C_4>0$ is independent of $\ep$. Then, by using \eqref{est_dif} and \eqref{bound_app},
we obtain that for sufficiently small $\ep$,
\begin{equation}\label{low_bound}\begin{split}
\|(U,B_1)(t,\cdot)\|_{W_\alpha^{0,\infty}}&\geq\|(U_\epsilon,B_{1,\epsilon})(t,\cdot)\|_{W_\alpha^{0,\infty}}
-\|(\tilde U,\tilde B_1)(t,\cdot)\|_{W_\alpha^{0,\infty}}\\
&\geq C_0^{-1}e^{\frac{\sigma_0t}{\sqrt{\ep}}}-C_4 \epsilon^{\frac{1}{4}-\mu}e^{\frac{\sigma_0t}{\sqrt{\ep}}}\big(1+\frac{t^4}{\epsilon}\big)\\
&\geq C_5e^{\frac{\sigma_0t}{\sqrt{\ep}}},
\end{split}\end{equation}
provided that $t\ll\epsilon^{1/4}$.
As $\sigma<\sigma_0,$ comparing \eqref{up_bound} with \eqref{low_bound},
the contradiction arises when $\frac{\mu}{\sigma_0-\sigma}|\ln \epsilon|\sqrt{\epsilon}\ll t\ll \epsilon^{1/4}$
with sufficiently small $\epsilon$. Thus, the proof of Theorem \ref{thm_lin} is completed.
\end{proof}

\begin{rem}
	Through the detailed calculation of the error terms $(r_\epsilon^1, r_\epsilon^2 )$ given by \eqref{lin_ap}, we know that the diffusion term $\partial_y^2 u_\epsilon$ generates the term $O(\epsilon^{-1/4} )$ in $R_\epsilon^1(t,y)$.  Thus, if we consider the MHD boundary layer problem with magnetic diffusion and the corresponding background magnetic field depends on the time variable, by applying a similar approach as above we can obtain the same result as Theorem \ref{thm_lin}.
\end{rem}

\bigskip

\appendix
\section{Derivation of Boundary Layer System} \label{A1}

In this section, we will give a formal derivation of the boundary layer equations \eqref{BLE}. When the magnetic Reynolds number is much larger than the hydrodynamics Reynolds number, the magnetic diffusion term can be ignored in the derivation. That is, we consider the following two dimensional incompressible MHD equations with a small viscosity coefficient in a domain with a flat boundary:
\begin{align}\label{MHD}
\left\{
\begin{array}{ll}
\partial_t\mathbf{u}^\varepsilon+(\mathbf u^\varepsilon\cdot\nabla)\mathbf u^\varepsilon+\nabla p^\varepsilon-(\mathbf b^\varepsilon\cdot\nabla) \mathbf b^\varepsilon=\varepsilon \triangle\mathbf u^\varepsilon,\\
\partial_t\mathbf b^\varepsilon+(\mathbf u^\varepsilon\cdot\nabla)\mathbf b^\varepsilon=(\mathbf b^\varepsilon\cdot\nabla) \mathbf u^\varepsilon,\\
\hbox{div}\mathbf u^\varepsilon=0,\quad \hbox{div}\mathbf b^{\varepsilon}=0,
\end{array}
\right.
\end{align}
where $t>0,~ {\bf x}=(x,y)\in\mathbb{T}\times\mathbb{R}_+$, $\mathbf u^\varepsilon=(u^\varepsilon_1, u^\varepsilon_2)$ the velocity, $\mathbf b^\varepsilon=(b^\varepsilon_1, b^\varepsilon_2)$ the magnetic field, $p^\varepsilon$ the total pressure: $$
	p^\varepsilon=p_f^\varepsilon+\frac{\left|\mathbf{b}^\varepsilon\right|^2}{2},$$ where $p_f^\varepsilon$ is the fluid pressure. The small parameter $\varepsilon>0$ denotes the viscosity coefficient.

The initial data for (\ref{MHD}) is given by
\begin{align}
\label{ID}
\mathbf u^\varepsilon(0,x,y)=\mathbf u_0(x,y),\qquad \mathbf b^\varepsilon(0,x,y)=\mathbf b_0(x,y).
\end{align}
The  no-slip boundary condition is imposed on velocity:
\begin{align}\label{bd-MHD}
\mathbf{u}^\varepsilon|_{y=0}=\mathbf0. 
\end{align}
Since  the no-slip boundary condition (\ref{bd-MHD}) is given for the velocity, there is no need to impose any boundary condition on magnetic field $\mathbf b^\varepsilon$, at least for the classical solutions. Indeed, the restriction of equations \eqref{MHD} on the boundary and using the boundary condition \eqref{bd-MHD} imply
\begin{align}\label{BC_b2}
	b_2^\varepsilon(t,x,y)|_{y=0}~\equiv~b_2^\varepsilon(0,x,y)|_{y=0}=b_{02}(x,0).
\end{align}

When the hydrodynamic Reynolds number tends to infinity, which corresponds to the viscocity coefficient $\varepsilon$ tending to zero, the related limited system of  (\ref{MHD}) is the ideal  incompressible MHD system:
\begin{align}
\label{IIV}
\left\{
\begin{array}{ll}
\partial_t\mathbf u^0+(\mathbf u^0\cdot\nabla)\mathbf u^0+\nabla p^0=(\mathbf b^0\cdot\nabla)\mathbf b^0,\\
\partial_t\mathbf b^0+(\mathbf u^0\cdot\nabla)\mathbf b^0=(\mathbf b^0\cdot\nabla)\mathbf u^0,\\
\hbox{div}\mathbf u^0=0,\quad \hbox{div}\mathbf b^{0}=0.
\end{array}
\right.
\end{align}
To avoid the initial layer in the study of the vanishing viscosity limit process for (\ref{MHD}), the initial data for (\ref{IIV}) is taken to be  the same as the initial data of viscous flow (\ref{ID}), \textit{i.e.},
\begin{align}
\label{IDI}
\mathbf u^0(0,x,y)=\mathbf u_0(x,y),\qquad \mathbf b^0(0,x,y)=\mathbf b_0(x,y).
\end{align}
For the well-posedness of  the ideal  incompressible MHD equations (\ref{IIV}), we may impose the following boundary condition for the normal components of velocity and magnetic field:
\begin{align}
\label{1.2I}
u^0_2|_{y=0}=0, \qquad b^0_2|_{y=0}=0.
\end{align}
Hence, it is necessary that the initial data $b_{02}(x,y)$ satisfies the compatibility condition
\begin{align*}
b_{02}(x,0)=0,
\end{align*}
which, together with \eqref{BC_b2}, imply
\begin{align}
\label{BCBB}
b^\varepsilon_{2}(t,x,0)=0.
\end{align}
Consequently,  we impose the boundary conditions  (\ref{bd-MHD}) and (\ref{BCBB}) for (\ref{MHD}).

\begin{rem}
	We can use the process in Section 2 to reformulate the viscous MHD equations \eqref{MHD}. Precisely,  consider the stream function $\psi^\varepsilon$ of the magnetic field $\mathbf{b}^\varepsilon$:
\begin{align*}
\partial_{y}\psi^\varepsilon =b_1^\varepsilon ,\qquad-\partial_x\psi^\varepsilon =b_2^\varepsilon,\qquad \psi^\varepsilon|_{y=0}=0.
\end{align*}
Then, by using $\psi^\varepsilon$ we can rewrite the equations \eqref{MHD} as 
\begin{align*}
	\begin{cases}
		\partial_t \psi^\varepsilon+(\mathbf{u}^\varepsilon\cdot\nabla)\psi^\varepsilon=0,\\
		\partial_t \mathbf{u}^\varepsilon+(\mathbf{u}^\varepsilon\cdot\nabla)\mathbf{u}^\varepsilon+\nabla p^\varepsilon_f+\triangle \psi^\varepsilon\nabla\psi^\varepsilon =\varepsilon\triangle \mathbf{u}^\varepsilon,
	\end{cases}
\end{align*}
which is exactly the two-dimensional incompressible Navier-Stokes-Korteweg equations. Therefore in the two-dimensional case the effect of magnetic field on the fluid can be regarded in some sense as some kind of capillarity; see also \cite{L-X-Z, R-W-X-Z}.
\end{rem}

To study the vanishing viscosity limit for (\ref{MHD})-(\ref{bd-MHD}) and (\ref{BCBB}), the asymptotic boundary layer expansion is an effective tool, which was proposed by L. Prandtl in his pioneering work \cite{P}. Moreover,  the well-posedness theory and the properties of solutions to the leading order nonlinear boundary layer equations play a key role in studying the vanishing viscosity limit.
To derive the related MHD boundary layer equations, we follow the boundary layer expansion in \cite{P} to write
\begin{align}
u^\varepsilon_1(t,x,y)&=u^0_1(t,x,y)+u^b_1(t,x,\tilde{y})+o(1),\label{ANSATZ1}\\
u^\varepsilon_2(t,x,y)&=u^0_2(t,x,y)+\sqrt{\varepsilon}\big(u_2^1(t,x,y)+u^b_2(t,x,\tilde{y})\big)+o(\sqrt{\varepsilon}),\label{ANSATZ2}\\
b^\varepsilon_1(t,x,y)&=b^0_1(t,x,y)+b^b_1(t,x,\tilde{y})+o(1),\label{ANSATZ3}\\
b^\varepsilon_2(t,x,y)&=b^0_2(t,x,y)+\sqrt{\varepsilon}\big(b_2^1(t,x,y)+b^b_2(t,x,\tilde{y})\big)+o(\sqrt{\varepsilon}),\label{ANSATZ4}\\
p^\varepsilon(t,x,y)&=p^0(t,x,y)+p^b(t,x,\tilde{y})+o(1),\label{ANSATZ7}
\end{align}
where the fast variable $\tilde y=\varepsilon^{-1/2}y$. 

Substituting the ansatz (\ref{ANSATZ1})-(\ref{ANSATZ7}) into (\ref{MHD}) and comparing terms  in each equation according to the order of $\varepsilon$,
$(u^0_1,u^0_2,b^0_1,b^0_2,p^0)(t,x,y)$ satisfies (\ref{IIV});  $(u_2^1, b_2^1)(t,x,y)$ represents
the next order of inner flow that  can be obtained by solving a system of linear ideal MHD equations; and $(u^b_1, \sqrt{\varepsilon}u^b_2, b^b_1, \sqrt{\varepsilon}b^b_2,p^b)$ represents
 the boundary layer that decays  to zero as $\tilde{y}$ tends to $+\infty$. Set
\begin{align*}
\left\{
\begin{array}{ll}
u(t,x, \tilde{y})=u^0_1(t,x,0)+u^b_1(t,x, \tilde{y}),\\
v(t,x, \tilde{y})=u_2^1(t,x,0)+u^b_2(t,x, \tilde{y})+\tilde{y}\partial_y u^0_2(t,x,0),\\
b_1(t,x, \tilde y)=b^0_1(t,x,0)+b^b_1(t,x, \tilde{y}),\\
b_2(t,x, \tilde y)=b_2^1(t,x,0)+b^b_2(t,x, \tilde{y})+\tilde{y}\partial_y b^0_2(t,x,0),\\
 p(t,x,\tilde{y})=p^0(t,x,0)+p^b(t,x, \tilde{y}),
\end{array}
\right.
\end{align*}
which correspond to  the leading order terms in  the expansion \eqref{ANSATZ1}-\eqref{ANSATZ7} with respect to $\varepsilon$.
 Here $(u,v,b_1,b_2,p)(t,x,\tilde y)$ satisfies
\begin{align}
\label{BL}
\left\{
\begin{array}{ll}
\partial_tu+u\partial_xu+v\partial_{\tilde y}u+\partial_xp=\partial_{\tilde y}^2u+b_1\partial_xb_1+b_2\partial_{\tilde y}b_1,\\
\partial_{\tilde y}p=0,\\
\partial_tb_1+u\partial_xb_1+v\partial_{\tilde y}b_1=b_1\partial_xu+b_2\partial_{\tilde y}u,\\
\partial_tb_2+u\partial_xb_2+v\partial_{\tilde y}b_2=b_1\partial_xv+b_2\partial_{\tilde y}v,\\
\partial_xu+\partial_{\tilde y}v=0,\quad \partial_xb_1+\partial_{\tilde y}b_2=0.
\end{array}
\right.
\end{align}
According to (\ref{ID}), (\ref{IDI}), (\ref{ANSATZ1}) and (\ref{ANSATZ3}), the initial data of (\ref{BL}) are given by
\begin{align}
\label{IVBLI}
u(0,x,\tilde{y})=0,\qquad b_1(0,x,\tilde{y})=b^0_{1}(0,x,0)=b_{01}(x,0),
\end{align}
where \eqref{IDI} and  the compatibility condition  $u_{01}(x,0)=0$ are used.

The boundary conditions and the far-field conditions are written as follows,
\begin{align}
\label{BBC}
u|_{\tilde y=0}=v|_{\tilde y=0}=b_2|_{\tilde y=0}=0,
\end{align}
and
\begin{align}
\label{FC}
\lim_{\tilde y\rightarrow\infty}u(t,x,\tilde y)=u_1^0(t,x,0):= 
U(t,x),\quad \lim_{\tilde y\rightarrow\infty}b_1(t,x,\tilde y)=b_1^0(t,x,0) :=  
B(t,x).
\end{align}

\begin{rem}
The boundary layer problem \eqref{BL}-\eqref{FC} also can be derived by using the scaling method proposed in \cite{O-S}. Indeed, one can choose the following scale transform:
\begin{align*}
 t=t,\quad x=x,\quad \tilde y=\varepsilon^{-\frac{1}{2}}y,
 \end{align*}
and
\begin{align*}
\left\{
\begin{array}{ll}
u(t,x,\tilde y)=u_1^\varepsilon(t,x, y),\\
v(t,x,\tilde y)=\varepsilon^{-\frac{1}{2}}u^\varepsilon_2(t,x, y),
\end{array}
\right.
\qquad p(t,x,\tilde y)=p^\varepsilon(t,x,y),
\end{align*}
\begin{align*}
\left\{
\begin{array}{ll}
b_1(t,x,\tilde y)=b_1^\varepsilon(t,x, y),\\
b_2(t,x,\tilde y)=\varepsilon^{-\frac{1}{2}}b_2^\varepsilon(t,x, y).
\end{array}
\right.
\end{align*}
\end{rem}

To simplify the system \eqref{BL},  note that  the second equation in $\eqref{BL}$ shows that there is no strong boundary layer for pressure, that is, $p^b(t,x,\tilde y)\equiv0$. As a consequence,
$$p(t,x,\tilde y)\equiv p^0(t,x,0):= 
(t,x).$$
In addition, the fourth equation in $(\ref{BL})$ is a direct consequence of the third and fifth equations in $(\ref{BL})$ and the boundary conditions $(\ref{BBC})$, at least it holds true for  smooth solutions. Therefore,
the system \eqref{BL} is reduced to
\begin{align}
\label{BLLL}
\left\{
\begin{array}{ll}
\partial_tu+u\partial_xu+v\partial_{\tilde y}u-\partial_{\tilde y}^2u-b_1\partial_xb_1-b_2\partial_{\tilde y}b_1=-\partial_x p ,\\
\partial_t b_1+u\partial_x b_1+v\partial_{\tilde y}b_1-b_1\partial_xu-b_2\partial_{\tilde y}u=0,\\
\partial_xu+\partial_{\tilde y}v=0,\quad \partial_xb_1+\partial_{\tilde y}b_2=0,\\
(u,v,b_2)|_{\tilde y=0}=0,\qquad \lim\limits_{\tilde y\to+\infty}(u,b_1)(t,x,\tilde y)=(U,B)(t,x),\\
(u,b_1)|_{t=0}=(u_0,b_0)(x,y).
\end{array}
\right.
\end{align}
Moreover,  the known functions $U, B$ and $p$ satisfy  Bernoulli's laws
\begin{align}
\label{MC1}
\left\{
\begin{array}{ll}
\partial_t U+U\partial_x U+\partial_x p=B\partial_x B,
\\
\partial_t B+U\partial_x B=B\partial_x U.
\end{array}
\right.
\end{align}

\bigskip

\section{Proof of Lemma \ref{MTV}}
Finally,  we give the proof of Lemma \ref{MTV}. 
\begin{proof}[Proof of Lemma \ref{MTV}]
We only prove the inequality \eqref{Moser} since the proof of \eqref{Moser1} is similar. The proof is divided into two cases.

{Case 1: $\alpha=\bf0$ or $\beta=\bf0.$ } Without loss of generality, we assume $\alpha=\bf0$. By the Sobolev embedding  and $|\beta|\leq m$,
\begin{align*}
	\|( u\cdot\partial^\beta v)(t,\cdot)\|_{L^2(\Omega)}\leq &\|u(t,\cdot)\|_{L^\infty(\Omega)}\|\partial^\beta v(t,\cdot)\|_{L^2(\Omega)}\\
	\lesssim &\|u(t)\|_{\mathcal{H}^2}\|v(t)\|_{\mathcal{H}^m},
\end{align*}
which implies \eqref{Moser} for $m\geq2$.

{Case 2: $|\alpha|\geq1$ and $|\beta|\geq1.$ }
	 From $|\alpha|+|\beta|\leq m$,  there exist $|\alpha|\leq m-1$ and $|\beta|\leq m-1$, so that the Sobolev embedding gives
	\begin{align*}
		\big\|\big(\partial^\alpha u\cdot \partial^{\beta}v\big)(t,\cdot)\big\|_{L^2(\Omega)}\leq&\big\|\partial^\alpha u(t,\cdot)\big\|_{L^4(\Omega)}\cdot \big\|\partial^{\beta}v(t,\cdot)\big\|_{L^4(\Omega)}\\
		\lesssim& \big\|\partial^\alpha u (t,\cdot)\big\|_{H^1(\Omega)}\|\partial^\alpha v(t)\|_{H^{1}(\Omega)}\\
		\lesssim &\big\|u(t)\big\|_{\mathcal{H}^{|\alpha|+1}}\|v(t)\|_{\mathcal{H}^{|\beta|+1}},
	\end{align*}
 which implies \eqref{Moser}. Hence,  we complete the proof of Lemma \ref{MTV}.

\end{proof}

\bigskip

 \section*{Acknowledgement}
The research of C.-J. Liu was sponsored by National Natural Science Foundation of China (Grant No. 11743009, 11801364), Shanghai Sailing Program (Grant No. 18YF1411700) and Shanghai Jiao Tong University (Grant No. WF220441906).
The research  of D. Wang was  partially supported by the
National Science Foundation under grants   DMS-1613213 and DMS-1907519.
The research of F. Xie was partially supported by National Natural Science Foundation of China (Grant No.11571231, 11831003) and the China Scholarship Council.
The research of T. Yang was partially supported by the General Research Fund of Hong Kong, CityU No.11302518.

\bigskip

\bibliographystyle{plain}

\end{document}